\documentclass[a4paper,reqno,10pt]{amsart}

\usepackage{epsf,graphicx,epsfig}
\usepackage{amscd}
\usepackage{amsmath,latexsym,amssymb,amsthm}
\usepackage[nospace,noadjust]{cite}
\usepackage{textcomp}
\usepackage{amsfonts}
\usepackage{setspace,cite}
\usepackage{lscape,fancyhdr,fancybox}
\usepackage{stmaryrd}
\usepackage{mathrsfs}
\usepackage[all,cmtip]{xy}
\usepackage{tikz}
\usepackage{cancel}
\usetikzlibrary{shapes,arrows,decorations.markings}
\usepackage{graphicx}
\usepackage{tikz-cd}
\usepackage{color}
\usepackage{url}
\usepackage{enumerate}
\usepackage[mathscr]{euscript}
  \theoremstyle{plain}

\swapnumbers
    \newtheorem{thm}{Theorem}[section]
    \newtheorem{prop}[thm]{Proposition}
     
   \newtheorem{lemma}[thm]{Lemma}

    \newtheorem{subsec}[thm]{}
\theoremstyle{definition}
    \newtheorem{defn}[thm]{Definition}
        \newtheorem{remark}[thm]{Remark}
    \newtheorem{exam}[thm]{Example}

\theoremstyle{remark}

\title{}
\author{}
\date{}
\usepackage{amssymb}

\usepackage{hyperref}
\hypersetup{
	colorlinks,
	citecolor=blue,
	filecolor=black,
	linkcolor=blue,
	urlcolor=black
}

\begin{document}

\title[Quasi-twilled associative algebras and deformation maps]{Quasi-twilled associative algebras, deformation maps and their governing algebras}

\author{Apurba Das}
\address{Department of Mathematics,
Indian Institute of Technology, Kharagpur 721302, West Bengal, India.}
\email{apurbadas348@gmail.com, apurbadas348@maths.iitkgp.ac.in}

\author{Ramkrishna Mandal}
\address{Department of Mathematics, Indian Institute of Technology, Kharagpur 721302, West Bengal, India.}
\email{ramkrishnamandal430@gmail.com}

\begin{abstract}
A quasi-twilled associative algebra is an associative algebra $\mathbb{A}$ whose underlying vector space has a decomposition $\mathbb{A} = A \oplus B$ such that $B \subset \mathbb{A}$ is a subalgebra. In the first part of this paper, we give the Maurer-Cartan characterization and introduce the cohomology of a quasi-twilled associative algebra.

In a quasi-twilled associative algebra $\mathbb{A}$, a linear map $D: A \rightarrow B$ is called a {\em strong deformation map} if $\mathrm{Gr}(D) \subset \mathbb{A}$ is a subalgebra. Such a map generalizes associative algebra homomorphisms, derivations, crossed homomorphisms and the associative analogue of modified {\sf r}-matrices. We introduce the cohomology of a {strong deformation map} $D$ unifying the cohomologies of all the operators mentioned above. We also define the governing algebra for the pair $(\mathbb{A}, D)$ to study simultaneous deformations of both $\mathbb{A}$ and $D$.

On the other hand, a linear map $r: B \rightarrow A$ is called a {\em weak deformation map} if $\mathrm{Gr} (r) \subset \mathbb{A}$ is a subalgebra. Such a map generalizes relative Rota-Baxter operators of any weight, twisted Rota-Baxter operators, Reynolds operators, left-averaging operators and right-averaging operators. Here we define the cohomology and governing algebra of a {weak deformation map} $r$ (that unify the cohomologies of all the operators mentioned above) and also for the pair $(\mathbb{A}, r)$ that govern simultaneous deformations.
\end{abstract}

\maketitle

%\curraddr{}
%\email{}

%\subjclass[2010]{}
%\keywords{}

\medskip

 
 %16D20, 16E40, 16S80, 16T25, 16W99, 17B38.

 
 %Rota-Baxter operators, twisted Rota-Baxter operators, Crossed homomorphism, Averaging operators, Reynolds operators, Cohomology, Deformation.

 %\medskip

%\noindent {\sf Date of resubmission:} July 26, 2021.

\thispagestyle{empty}

\tableofcontents

%\vspace{0.2cm}

\medskip

\section{Introduction}

The cohomology of an algebraic structure is an important invariant that
determine extensions and formal deformations of the structure. Hochschild \cite{hoch} first discovered the cohomology of associative algebras in the study of extension problems. Later, Gerstenhaber \cite{gers} developed formal deformations of an associative algebra $A$ and showed that such deformations are closely related to the Hochschild cohomology groups of $A$. Subsequently, Nijenhuis and Richardson \cite{nij-ric} generalized the formal deformation theory to Lie algebras. After that, many authors developed the cohomology and deformation theory of various other algebraic structures over the years. In particular, Balavoine \cite{bala} generalized the theories to algebras over binary quadratic operads. See also \cite{lod-val-book} for cohomology and deformations of algebras over more general operads.

\medskip

Algebraic structures are often characterized by certain operators defined on them. Among others, algebra homomorphisms, derivations and crossed homomorphisms are the most popular. Another interesting operator that appears in the solutions of the modified associative Yang-Baxter equation is given by the (associative analogue of) modified {\sf r}-matrix \cite{guo-commun,das-modified}.
On the other hand, to study the splitting of algebras and duplications of algebra operations, one encounters some operators defined on them, e.g. Rota-Baxter operators (of weight $0$) \cite{guo-book}, Rota-Baxter operators of weight $1$ \cite{Bai-Guo-Ni,guo-commun}, twisted Rota-Baxter operators \cite{uchino}, Reynolds operators \cite{zhang-gao-guo}, left averaging operators \cite{cao} and right averaging operators \cite{cao}. In an associative algebra $A$ with the multiplication $\cdot$ ~, all the operators as mentioned above are defined respectively by the following:
\begin{align*}
& \text{algebra homomorphism }    && D (a \cdot b) = D (a) \cdot D (b), \\
& \text{derivation} && D (a \cdot b) = D (a) \cdot b + a \cdot D(b), \\
& \text{crossed homomorphism} && D (a \cdot b) = D (a) \cdot b + a \cdot D(b) + D(a) \cdot D(b),\\
& \text{modified associative {\sf r}-matrix} && D (a) \cdot D(b) = D ( D(a) \cdot b + a \cdot D(b) ) - a \cdot b, \\
& \text{Rota-Baxter operator (of weight $0$)} && r (a) \cdot r(b) = r (    r(a) \cdot b + a \cdot r (b) ), \\
& \text{Rota-Baxter operator of weight $1$} && r (a) \cdot r(b) = r (    r(a) \cdot b + a \cdot r (b) + a \cdot b ),\\
& \text{$\theta$-twisted Rota-Baxter operator} && r(a) \cdot r(b) = r \big(    r(a) \cdot b + a \cdot r (b) + \theta ( r(a), r(b)) \big), \\
& \text{Reynolds operator} && r (a) \cdot r(b) = r (    r(a) \cdot b + a \cdot r (b) - r(a) \cdot r(b) ),\\
& \text{left averaging operator} && r(a) \cdot r(b) = r ( r(a) \cdot b ), \\
& \text{right averaging operator} && r (a) \cdot r(b) = r ( a \cdot r(b) ).
\end{align*}
Note that some of these operators also have their relative version, i.e. they can be defined with respect to suitable $A$-bimodules. In \cite{nij-ric2} Nijenhuis and Richardson initiate the study of cohomology and deformations of an operator. Namely, they defined and studied the cohomology of a Lie algebra homomorphism. Then Gerstenhaber and Schack \cite{gers-sch} established the cohomology of an associative algebra homomorphism and study deformations. The underlying higher structures for controlling the underlying algebras and homomorphisms were formulated in \cite{yau} and \cite{fre}. Recently, cohomology and deformation theories of derivations \cite{loday-morph}, crossed homomorphisms \cite{das-crossed,pei}, modified associative {\sf r}-matrices \cite{das-modified}, (relative) Rota-Baxter operators of weight $0$ \cite{tang,das-rota}, (relative) Rota-Baxter operators of weight $1$ \cite{das-weighted}, twisted Rota-Baxter operators \cite{das-twisted,das-ns}, Reynolds operators \cite{das-twisted,das-ns}, left averaging operators \cite{wang-zhou} and right averaging operators \cite{wang-zhou} were extensively studied. Note that the simultaneous deformations of the underlying algebras and the respective operators were also investigated in some cases \cite{das-mishra,gers-sch,das-avg,lazarev,fre,yau,wang-zhou,wang-zhou2}.

\medskip

Our main aim in this paper is to find a unified approach to study all the operators mentioned above. To do so, we emphasise that all the operators listed above can be described by their graphs which are subalgebras in some bigger algebras. To understand these bigger algebras in full generality, we consider the notion of a quasi-twilled associative algebra \cite{uchino-t} (also called an {\em extending datum} in \cite{agore2}). This is a quasi-analogue of a matched pair of associative algebras  \cite{agore} and the associative counterpart of a quasi-twilled Lie algebra \cite{jiang-sheng-tang}. Implicitly, a quasi-twilled associative algebra is an associative algebra $\mathbb{A}$ whose underlying vector space has a direct sum decomposition $\mathbb{A} = A \oplus B$ such that $B$ is a subalgebra of $\mathbb{A}$ (see Theorem \ref{explicit-qtaa} for the explicit description). Examples of quasi-twilled associative algebras include the direct product of two associative algebras, semi-direct products, semi-direct products twisted by Hochschild $2$-cocycles and of course matched pairs of associative algebras. In the first part of this paper, we give the Maurer-Cartan characterization of a quasi-twilled associative algebra $\mathbb{A}$ and as an application, we define the cohomology of $\mathbb{A}$. In particular, we get the Maurer-Cartan characterization and cohomology of a matched pair of associative algebras. To our knowledge, this cohomology was even unknown in the literature.

\medskip

In a matched pair of associative algebras, Agore \cite{agore} introduced the notion of a ``deformation map'' while studying the classifying compliments problem (CCP). Since a quasi-twilled associative algebra $\mathbb{A} =  A \oplus B$ is not symmetric with respect to $A$ and $B$, we obtain two variants of the deformation map in a quasi-twilled associative algebra $\mathbb{A}$. 

\medskip

A linear map $D: A \rightarrow B$ is said to be a {\em strong deformation map} if its graph $\mathrm{Gr} (D) \subset \mathbb{A}$ is a subalgebra. Such a map unifies algebra homomorphisms, derivations, crossed homomorphisms and the associative modified {\sf r}-matrices. A strong deformation map $D: A \rightarrow B$ induces an associative algebra structure on $A$ (denoted by $A_D$). Moreover, the pair $(A_D, B)$ forms a matched pair of associative algebras. As a result, one obtains an $A_D$-bimodule structure on the vector space $B$. We define the corresponding Hochschild cohomology of $A_D$ with coefficients in $B$ as the cohomology of the strong deformation map $D$. This cohomology generalizes the existing cohomologies of algebra homomorphisms \cite{gers-sch}, derivations \cite{loday-morph} and crossed homomorphisms \cite{das-crossed}. Next, given a quasi-twilled associative algebra $\mathbb{A}$, we characterize strong deformation maps by Maurer-Cartan elements in a suitable curved $L_\infty$-algebra, called the {\em controlling algebra} for strong deformation maps. Twisting the controlling algebra by a Maurer-Cartan element (i.e. a strong deformation map $D$), we obtain the $L_\infty$-algebra that governs the deformations of the operator $D$ keeping the underlying quasi-twilled associative algebra intact. Next, we proceed as follows to understand the simultaneous deformations of $\mathbb{A}$ and the operator $D$. For any two vector spaces $A$ and $B$ (not necessarily equipped with any additional structures), we first construct an $L_\infty$-algebra whose Maurer-Cartan elements correspond to pairs $(\mathbb{A}, D)$, where $\mathbb{A} = A \oplus B$ has a quasi-twilled associative algebra structure and $D: A \rightarrow B$ is a strong deformation map in $\mathbb{A}$. Again, twisting the $L_\infty$-algebra by a Maurer-Cartan element $(\mathbb{A}, D)$, one gets the governing $L_\infty$-algebra of the pair $(\mathbb{A}, D)$ of a quasi-twilled associative algebra $\mathbb{A}$ and a strong deformation map $D$ in it. This higher structure governs the simultaneous deformations of $\mathbb{A}$ and $D$.

\medskip

On the other hand, a linear map $r: B \rightarrow A$ in a quasi-twilled associative algebra $\mathbb{A}$ is said to be a {\em weak deformation map} if its graph $\mathrm{Gr }(r) \subset \mathbb{A}$ is a subalgebra. It unifies (relative) Rota-Baxter operators of weight $0$, (relative) Rota-Baxter operators of weight $1$, twisted Rota-Baxter operators, Reynolds operators, left averaging operators and right averaging operators. We observe that a weak deformation map $r$ induces a new quasi-twilled associative algebra structure $\mathbb{A}_r$ on the same direct sum $A \oplus B$. However, this induced structure $\mathbb{A}_r$ need not form a matched pair of associative algebras. (For this reason, we call $r$ a weak deformation map and $D$ a strong deformation map.) Note that, the quasi-twilled associative structure $\mathbb{A}_r$ gives rise to a new associative algebra structure on $B$ (denoted by $B_r$) and a $B_r$-bimodule structure on $A$. We call the corresponding Hochschild cohomology of $B_r$ with coefficients in $A$ as the cohomology of the weak deformation map $r$. As expected, our cohomology generalizes the existing cohomologies of (relative) Rota-Baxter operators of weight $0$ \cite{das-rota}, (relative) Rota-Baxter operators of weight $1$ \cite{das-weighted}, twisted Rota-Baxter operators \cite{das-ns}, Reynolds operators \cite{das-ns}, left averaging operators \cite{wang-zhou} and right averaging operators \cite{wang-zhou}. Given a quasi-twilled associative algebra $\mathbb{A}$, we also construct the {\em controlling algebra} (which is an $L_\infty$-algebra) for weak deformation maps. As before, twisting this controlling algebra by a Maurer-Cartan element (i.e. a weak deformation map $r$), we get the $L_\infty$-algebra that governs the deformations of the operator $r$. In a particular case, we obtain the controlling algebra for ``deformation maps'' in a matched pair of associative algebras introduced by Agore \cite{agore}. In the end, we also construct the governing $L_\infty$-algebra for a pair $(\mathbb{A}, r)$ of a quasi-twilled associative algebra $\mathbb{A}$ and a weak deformation map $r$ to study the simultaneous deformations. This structure generalizes the existing governing $L_\infty$-algebras for Rota-Baxter algebras \cite{das-mishra,wang-zhou2} and left (resp. right) averaging algebras \cite{das-avg,wang-zhou}.

\medskip

In \cite{agore4} the authors have introduced non-abelian $2$-cocycles on an associative algebra $A$ with values in another associative algebra $B$ while studying non-abelian extensions. In the final part of this paper, we observe that a non-abelian $2$-cocycle on $A$ with values in $B$ also gives rise to a quasi-twilled associative algebra. Then we define Rota-Baxter operators twisted by a non-abelian $2$-cocycle and a new algebraic structure, called {\em twisted tridendriform algebra}. We show that any Rota-Baxter operator twisted by a non-abelian $2$-cocycle induces a twisted tridendriform algebra and conversely, any twisted tridendriform algebra arises in this way. This unifies the facts that any Rota-Baxter operator of weight $1$ induces a tridendriform algebra \cite{ebrahimi} and any twisted Rota-Baxter operator induces an NS-algebra structure \cite{uchino}.

\medskip

It is important to remark that the authors of \cite{jiang-sheng-tang} have recently studied deformation maps in a quasi-twilled Lie algebra. They considered two types of deformation maps (type I and type II) generalizing deformation maps in a matched pair of Lie algebras \cite{agore3}. Among others, they mainly studied cohomology and deformations of these maps. Motivated by the work of \cite{jiang-sheng-tang}, in the present paper, we have considered two types of deformation maps in a quasi-twilled associative algebra. However, we call them strong deformation maps and weak deformation maps respectively (instead of type I and type II deformation maps) depending on whether the induced quasi-twilled associative structure on $\mathbb{A}$ gives rise to a matched pair of associative algebras or not. In addition to the cohomology and deformations of these maps, we extensively study quasi-twilled associative algebras and simultaneous deformations of a quasi-twilled associative algebra with a given (strong or weak) deformation map. Moreover, we discuss Rota-Baxter operators twisted by non-abelian $2$-cocycles and the induced algebraic structure.

\medskip

The paper is organized as follows. In section \ref{sec2}, we recall some necessary background on (curved) $L_\infty$-algebras and their construction by derived brackets. In section \ref{sec3}, we consider quasi-twilled associative algebras and give their Maurer-Cartan characterization and cohomology. We define strong deformation maps in a quasi-twilled associative algebra $\mathbb{A}$ and introduce their cohomology in section \ref{sec4}. We study the governing $L_\infty$-algebras and deformations of a strong deformation map $D$ and of the pair $(\mathbb{A}, D)$ in section \ref{sec5}. We complete the analogous study for weak deformation maps in sections \ref{sec6} and \ref{sec7}, respectively. Finally, in section \ref{sec8}, we study Rota-Baxter operators twisted by any non-abelian $2$-cocycle and consider twisted tridendriform algebras.

%\textcolor{red}{Koszul sign}

\medskip

\section{Some background on (curved) L-{infinity}-algebras}\label{sec2}
This section recalls (curved) $L_{\infty}$-algebras and their Maurer-Cartan elements. In particular, we recall Voronov's construction of a (curved) $L_{\infty}$-algebra whose multi-ary operations are given by derived brackets. Our main references are \cite{getzler,kaj-stas,lada-markl,lazarev,voro}.
\begin{defn}
 A {\bf curved $L_\infty$-algebra} is a pair $(\mathcal{G},\{l_k\}_{k=0}^{\infty})$ consisting of a graded vector space $\mathcal{G} = \displaystyle\oplus_{n\in \mathbb{Z}}~ \mathcal{G}_n$ equipped with a collection $\{l_k:{\mathcal{G}}^{\otimes k}\rightarrow \mathcal{G} \}_{k=0}^{\infty}$ of graded linear maps of degree $1$ satisfying the following properties: 
\begin{itemize}
    \item [(i)] for each $k\geq 1$, the map $l_k$ is graded symmetric in the sense that \begin{center}
        ${l_k}(x_{\sigma (1)},\ldots,x_{\sigma(k)}) = \epsilon(\sigma)~ {l_k}(x_1,\ldots,x_k)$, for all $x_1,\ldots,x_k\in \mathcal G$ and $\sigma \in S_k$,
    \end{center}
 \item [(ii)] for any $N\geq 0$, we have
  \begin{align}\label{a1}
\displaystyle\sum^{N}_{i=0}~\sum_{\sigma\in Sh(i,n-i)} \epsilon(\sigma)~l_{n-i+1}\big(l_i(x_{\sigma(1)},\ldots,x_{\sigma(i)}),x_{\sigma(i+1)},\ldots,x_{\sigma(N)}\big) = 0,
 \end{align}
 for all homogeneous elements $x_1,\ldots,x_N \in \mathcal{G}$.
\end{itemize}
\end{defn}
In the above definition, $\epsilon (\sigma)$ is the {\em Koszul sign} that appears in the graded context. Also, note that $l_0$ represents a degree 1 element of $\mathcal G$. When $l_0=0$, a curved $L_\infty$-algebra becomes an $L_\infty$-algebra.

A curved $L_\infty$-algebra  $(\mathcal{G},\{l_k\}_{k=0}^{\infty})$ is said to be {\bf weakly filtered} if there exists a descending filtration ${\mathcal{F}}_{\bullet}\mathcal{G}$ of $\mathcal{G}$ (that is, $\mathcal{G}= {\mathcal{F}}_1\mathcal{G}\supset {\mathcal{F}}_2\mathcal{G} \supset \cdots  \supset {\mathcal{F}}_k\mathcal{G}\supset \cdots $) and a natural number $N$ (called an index) such that 
\begin{align*}
     \mathcal{G} = \lim_{k \to \infty} {\mathcal{G}/{\mathcal{F}}_k\mathcal{G}} ~~ \text{ and } ~~ l_k(\mathcal{G},\ldots,\mathcal{G}) \subset {\mathcal{F}}_k\mathcal{G}, \text{ for all } k\geq N.
    \end{align*}
In the rest of this paper, we assume that all (curved) $L_{\infty}$-algebras are weakly filtered so that all formal sums under consideration are convergent.

 Let $(\mathcal{G},\{l_k\}_{k=0}^{\infty})$ be a curved $L_{\infty}$-algebra. A degree 0 element $\alpha\in \mathcal{G}_0$ is said to be a {\bf Maurer-Cartan element} of the curved $L_{\infty}$-algebra if it satisfies the following Maurer-Cartan equation \begin{align*}
 l_0 + \displaystyle\sum_{k=1}^{\infty} \dfrac{1}{k!}~l_k(\alpha,\ldots,\alpha) = 0.
  \end{align*}
\begin{thm}\label{Tmcla}
 Let $\alpha$ be a Maurer-Cartan element of the curved $L_{\infty}$-algebra $(\mathcal{G},\{l_k\}_{k=0}^{\infty})$. For each $k\geq 1$, we define a degree 1 graded linear map ${l^\alpha_k}:{\mathcal{G}}^{\otimes k}\rightarrow \mathcal{G}$ by 
 \begin{align*}
 {l^\alpha_k}(x_1,\ldots,x_k) :=  \displaystyle\sum_{n=0}^{\infty} \dfrac{1}{n!}~l_{n+k}(\alpha,\ldots,\alpha,x_1,\ldots,x_k), \text{ for } x_1,\ldots, x_k\in \mathcal{G}.   \end{align*} 
 Then $(\mathcal{G},\{l_k^\alpha \}_{k=1}^{\infty})$ is an $L_{\infty}$-algebra.
\end{thm}
The $L_{\infty}$-algebra $(\mathcal{G},\{l_k^\alpha\}_{k=1}^{\infty})$ constructed in the above theorem is said to be obtained from the curved $L_{\infty}$-algebra $(\mathcal{G},\{l_k\}_{k=0}^{\infty})$ twisted by the Maurer-Cartan element $\alpha$.

In \cite{voro} Voronov gave a construction of curved $L_{\infty}$-algebras which we shall recall now. We first need the following notion considered by Voronov.

\begin{defn}
A {\bf curved $V$-data} is a quadruple $(\mathfrak{g},\mathfrak{a},P,\Delta)$ consisting of a graded Lie algebra $(\mathfrak{g},[~,~])$, an abelian graded Lie subalgebra $\mathfrak{a}\subset \mathfrak{g}$, a projection map $P : \mathfrak{g}\rightarrow \mathfrak{a}\subset \mathfrak{g}$ such that $\mathrm{ker}(P)\subset \mathfrak{g}$ is a graded Lie subalgebra and an element $\Delta\in \mathfrak{g}_1$ that satisfies $[\Delta,\Delta] =0$.  
\end{defn}
In a curved $V$-data $(\mathfrak{g},\mathfrak{a},P,\Delta)$, if $\Delta\in \mathrm{ker}{(P)}_1$ then we simly call it a $V$-data.
\begin{thm} \label{Tcvd}
Let $(\mathfrak{g},\mathfrak{a},P,\Delta)$ be a curved $V$-data. \begin{itemize}
    \item [(i)] Then $(\mathfrak{a},\{l_k\}_{k=0}^{\infty})$ is a curved $L_{\infty}$-algebra, where \begin{align*} l_0= P(\Delta) \quad \text{ and } \quad 
  l_k(a_1,\ldots,a_k) = P[\cdots[[\Delta,a_1],a_2],\ldots,a_k], \text{ for } a_1,\ldots,a_k \in \mathfrak{a}.    
  \end{align*}
  \item[(ii)] If $\mathfrak{h}\subset \mathfrak{g}$ is a graded Lie subalgebra so that $[\Delta,\mathfrak{h}]\subset \mathfrak{h}$, then the graded vector space $\mathfrak{h}[1]\oplus \mathfrak{a}$ can also be equipped with a curved $L_{\infty}$-algebra structure whose structural operations $\{\widetilde{l}_k\}_{k=0}^{\infty}$ are given by
  \begin{align*} \widetilde{l}_0=~& P(\Delta),\\
\widetilde{l}_1(x[1]) =~& [\Delta,x][1]+P(x),\\ 
\widetilde{l}_2(x[1],y[1]) =~& (-1)^{\lvert x \rvert} [x,y][1] ,\\
\widetilde{l}_k\big(x[1],a_1,\ldots,a_{k-1}\big) =~& P~[\cdots [[x,a_1],a_2],\ldots,a_{k-1} ],~~k\geq 2,\\
\widetilde{l}_k\big(a_1,\ldots,a_{k}\big) =~& P~[\cdots [[\Delta,a_1],a_2],\ldots,a_{k} ],~~ k\geq 1,
\end{align*} 
\end{itemize}
for homogeneous elements $x,y\in \mathfrak{h}$ (considered as elements $x[1],y[1]\in \mathfrak{h}[1]$ with a degree shift) and homogeneous elements $a_1,\ldots,a_k\in \mathfrak{a}$. Up to the permutations of the above entries,
all the other linear maps vanish. Moreover, we have $\mathfrak{a}\subset \mathfrak{h}[1]\oplus \mathfrak{a}$ is a curved $L_{\infty}$-subalgebra.
\end{thm}

\medskip

\section{Quasi-twilled associative algebras}\label{sec3}

%The notion of a matched pair of associative algebras (also known as a twilled associative algebra) is a key object in associative bialgebra theory. 
In this section, we consider quasi-twilled associative algebras \cite{uchino-t,agore2} which are the quasi-analogue of matched pairs of associative algebras. We give the Maurer-Cartan characterization of quasi-twilled associative algebras and as a result, we define their cohomology theory.

%\begin{defn}
%    A {\bf matched pair of associative algebras}  \begin{itemize}
%        \item[(i)] 
%    \item[(iii)] For all $a,b\in A$, $ x,y\in B$, the following compatibility conditions are hold: 
%\begin{align*}
 %         a\vartriangleright (x \cdot_B y) =&~ (a\vartriangleright x)\cdot_B y + (a\leftharpoonup x)\vartriangleright  y,\\
%(x \cdot_B y) \vartriangleleft a =&~ x \cdot_B(y\vartriangleleft a) + x\vartriangleleft (y \rightharpoonup a),\\
%x\vartriangleleft(a\leftharpoonup y) + x \cdot_B (a \vartriangleright y) =&~ (x\rightharpoonup a) \vartriangleright y + (x\vartriangleleft a)\cdot_B y, \\
%x \rightharpoonup (a\cdot _A b) =&~ (x\rightharpoonup a)\cdot_A b + (x\vartriangleleft a) \rightharpoonup b,\\
%(a\cdot _A b)\leftharpoonup x =&~ a \cdot_A (b\leftharpoonup x) + a \leftharpoonup (b \vartriangleright x),\\
%a \leftharpoonup (x \vartriangleleft b) + a \cdot_A (x\rightharpoonup b) =&~ (a\vartriangleright x) \rightharpoonup b + (a\leftharpoonup x)\cdot_A b.
% \end{align*}
% \end{itemize}
%\end{defn}
%Let $(A, B, \vartriangleright, \vartriangleleft, \rightharpoonup, \leftharpoonup )$ be a matched pair of associative algebras. Then the direct sum $A\oplus B$ can be given an associative algebra structure with the multiplication 

 %\begin{remark}

%\medskip

%In the following, we consider the notion of a quasi-twilled associative algebra. This is a weaker notion than a twilled associative algebra where we only assume that $B$ is a subalgebra of an associative algebra structure on $A\oplus B$. 
\begin{defn} \cite{uchino-t,agore2}\label{dqtal}
    A {\bf quasi-twilled associative algebra} is an associative algebra $(\mathbb{A},\cdot_\mathbb{A})$ whose underlying vector space $\mathbb{A}$ has a direct sum decomposition $\mathbb{A} = A \oplus B$ such that $B$ is a subalgebra of $(\mathbb{A},\cdot_\mathbb{A})$. 
\end{defn}

%This definition is equivalent to the following one. 

\begin{thm}\label{explicit-qtaa}
A quasi-twilled associative algebra is equivalent to a tuple $(A_\mu, B_\nu, \vartriangleright, \vartriangleleft, \rightharpoonup, \leftharpoonup, \theta )$ in which 

\begin{itemize}
    \item $A$ is a vector space endowed with a bilinear map $\mu : A \times A \rightarrow A, ~ (a, b) \mapsto a \cdot_A b$,
    \item $B$ is a vector space endowed with a bilinear associative multiplication $\nu : B \times B \rightarrow B, ~ (x, y) \mapsto x \cdot_B y$,
\end{itemize}
and
\begin{align*}
    \vartriangleright ~ : A \times B\rightarrow B, \quad  \vartriangleleft ~ : B \times A\rightarrow B, \quad \rightharpoonup ~ : B \times A\rightarrow A, \quad  \leftharpoonup ~ : A \times B\rightarrow A , \quad \theta: A \times A\rightarrow B  
    \end{align*} 
    are bilinear maps such that $(A, \rightharpoonup, \leftharpoonup)$ is a bimodule over the associative algebra $B_\nu$ and for all $a,b,c\in A$, $ x,y\in B$ the following compatibility conditions are hold:
 \begin{align} \label{b1}
    a\cdot_A(b\cdot_A c) + a \leftharpoonup \theta(b,c)=&~
(a\cdot_A b)\cdot_A c +  \theta(a,b)\rightharpoonup c,\\ \label{b2}
(a\cdot_A b)\vartriangleright x + \theta(a,b)\cdot_B x =&~ a\vartriangleright (b\vartriangleright x)+ \theta(a,b\leftharpoonup x),\\ \label{b3}
x\vartriangleleft (a\cdot_A b) + x\cdot_B\theta(a,b) =&~ (x\vartriangleleft a)\vartriangleleft b+ \theta(x\rightharpoonup a, b),\\ \label{b4}
a\vartriangleright (x\vartriangleleft b) + \theta(a,x\rightharpoonup b)=&~(a\vartriangleright x)\vartriangleleft b + \theta(a\leftharpoonup x, b),\\ \label{b5}
 a\vartriangleright (x \cdot_B y) =&~ (a\vartriangleright x)\cdot_B y + (a\leftharpoonup x)\vartriangleright  y,\\ \label{b6}
(x \cdot_B y) \vartriangleleft a =&~ x \cdot_B(y\vartriangleleft a) + x\vartriangleleft (y \rightharpoonup a),\\ \label{b7}
x\vartriangleleft(a\leftharpoonup y) + x \cdot_B (a \vartriangleright y) =&~ (x\rightharpoonup a) \vartriangleright y + (x\vartriangleleft a)\cdot_B y, \\ \label{b8}
x \rightharpoonup (a\cdot _A b) =&~ (x\rightharpoonup a)\cdot_A b + (x\vartriangleleft a) \rightharpoonup b,\\ \label{b9}
(a\cdot _A b)\leftharpoonup x =&~ a \cdot_A (b\leftharpoonup x) + a \leftharpoonup (b \vartriangleright x),\\ \label{b10}
a \leftharpoonup (x \vartriangleleft b) + a \cdot_A (x\rightharpoonup b) =&~ (a\vartriangleright x) \rightharpoonup b + (a\leftharpoonup x)\cdot_A b,\\ \label{b11}
\theta(a,b\cdot_A c) + a\vartriangleright \theta(b,c)=&~ \theta(a\cdot_A b,c)+ \theta(a,b)\vartriangleleft c.
 \end{align}
\end{thm}

\begin{proof}
    Let $(\mathbb{A}, \cdot_\mathbb{A})$ be a quasi-twilled associative algebra with a decomposition $\mathbb{A} = A \oplus B$. Since $B$ is a subalgebra of $(\mathbb{A}, \cdot_\mathbb{A})$, it follows that the multiplication $~\cdot_\mathbb{A}~$ is of the form
    \begin{align}\label{omega-a-map}
(a,x) \cdot_\mathbb{A} (b, y) = \big( a \cdot_A b + a \leftharpoonup y + x \rightharpoonup b ~ \!,~ \! x \cdot_B y + a \vartriangleright y + x \vartriangleleft b + \theta (a, b)  \big),
\end{align}
for some bilinear maps $\mu, \nu, \vartriangleright, \vartriangleleft, \rightharpoonup, \leftharpoonup , \theta$. Finally, $~\cdot_\mathbb{A}~$ is associative if and only if the above bilinear maps satisfy all the conditions in the Theorem.
\end{proof}

\medskip

%It follows that any twilled associative algebra is a particular case of a quasi-twilled associative algebra in which $\theta = 0.$ 

%In the terminology of Definition \ref{dqtal}, a matched pair of associative algebras is a quasi-twilled associative algebra $(\mathbb{A},\cdot_\mathbb{A})$ in which $A$ and $B$ are both subalgebras of $\mathbb{A} = A \oplus B$.

%There are many examples of matched pairs of associative algebras (hence they are examples of quasi-twilled associative algebras) that appeared naturally in many contexts. These are listed below. We will come back to these examples when we define deformations maps in quasi-twilled associative algebras.

There are various examples of quasi-twilled associative algebras that are important for our paper. We will return to these examples when we define deformations maps in quasi-twilled associative algebras.

\begin{exam}\label{ex1}
  Let $A$ and $B$ be two associative algebras. Then the direct sum $A\oplus B$ endowed with the direct product 
  \begin{align*}
      (a,x)\odot (b,y) := (a \cdot_A b,~ \!  x \cdot_B y),~~~~\mathrm{ for~all~} (a,x),(b,y)\in A\oplus B
  \end{align*} 
  is an associative algebra. Then $(A \oplus B, \odot)$ is a quasi-twilled associative algebra.
\end{exam}
\begin{exam}\label{ex2}
 Let $A$ be an associative algebra and $(B, \vartriangleright, \vartriangleleft)$ be an $A$-bimodule. Consider the direct sum $A\oplus B$ with the semidirect product operation \begin{align*}
      (a,x)\ltimes (b,y) := (a \cdot_A b,~\!  a \vartriangleright y + x\vartriangleleft b ),~~~~\mathrm{ for~all~} (a,x),(b,y)\in A\oplus B.
  \end{align*} 
  Then $(A\oplus B,\ltimes)$ is a quasi-twilled associative algebra.
\end{exam} 
\begin{exam}\label{ex3}
Let $A$ be an associative algebra and $B$ be an associative algebra object in the category of $A$-bimodules. Then $A\oplus B$ endowed with the product \begin{align*}
      (a,x) \odot_\ltimes (b,y) := (a \cdot_A b,~ \!  x \cdot_B y+a \vartriangleright y + x\vartriangleleft b ),~~~~\mathrm{ for~all~} (a,x),(b,y)\in A\oplus B
  \end{align*} 
  is an associative algebra. This forms a quasi-twilled associative algebra.
\end{exam}
\begin{exam}\label{ex4}
 Let $A$ be an associative algebra and $(B, \vartriangleright)$ be a left $A$-module. Then $(A\oplus B,\cdot_\vartriangleright)$ is a quasi-twilled associative algebra, where  \begin{align*}
      (a,x)\cdot_\vartriangleright (b,y) := (a \cdot_A b,~ \!  a \vartriangleright y ),~~~~\mathrm{ for~all~} (a,x),(b,y)\in A\oplus B.
  \end{align*}   
\end{exam}
\begin{exam}\label{ex5}
 Let $A$ be an associative algebra and $(B, \vartriangleleft)$ be a right $A$-module. Then $(A\oplus B,\cdot_\vartriangleleft)$ is a quasi-twilled associative algebra, where  \begin{align*}
      (a,x)\cdot_\vartriangleleft (b,y) := (a \cdot_A b,~ \! x \vartriangleleft b ),~~~~\mathrm{ for~all~} (a,x),(b,y)\in A\oplus B.
  \end{align*}   
\end{exam}

\begin{exam}\label{matched-quasi}
   A {\em matched pair of associative algebras} \cite{agore,agore2}  is a tuple $(A_\mu, B_\nu, \vartriangleright, \vartriangleleft, \rightharpoonup, \leftharpoonup )$ consisting of two associative algebras $A_\mu$ and $B_\nu$ equipped with four bilinear maps \begin{align*}
    \vartriangleright ~ : A \times B\rightarrow B, \quad \vartriangleleft ~: B \times A\rightarrow B, \quad \rightharpoonup  ~: B \times A\rightarrow A, \quad\leftharpoonup  ~: A \times B\rightarrow A   
    \end{align*}
    such that $(B,\vartriangleright, \vartriangleleft)$ is an $A_\mu$-bimodule, $(A,\rightharpoonup, \leftharpoonup)$ is a $B_\nu$-bimodule and a set of compatibility conditions hold which ensure that the direct sum $A \oplus B$ endowed with the multiplication
\begin{align*}
    (a,x) \Join (b,y): = \big(a\cdot_A b+ a\leftharpoonup y + x \rightharpoonup b ~ \! ,~ \! x\cdot_B y + a \vartriangleright y + x \vartriangleleft b\big), \text{ for } (a,x), (b,y) \in A\oplus B
\end{align*}
 is an associative algebra. It follows that $(A\oplus B,{\Join})$ is a quasi-twilled associative algebra.

 Conversely, a matched pair of associative algebras is a quasi-twilled associative algebra $(\mathbb{A},\cdot_\mathbb{A})$ in which $A$ and $B$ are both subalgebras of $\mathbb{A} = A \oplus B$.
\end{exam}
 
 %Moreover, it follows from (\ref{a2}) that both $A$ and $B$ are subalgebras of $(A \oplus B, \Join)$. 
 % Conversely, let $A$ and $B$ be two vector spaces. Suppose there is an associative algebra structure $\Join$ on the direct sum $A\oplus B$ for which $A$ and $B$ are both subalgebras. Then there exist bilinear maps $\vartriangleright ~: A \times B\rightarrow B,~~ \vartriangleleft  ~ : B \times A\rightarrow B, ~~ \rightharpoonup ~ : B \times A\rightarrow A,~~ \leftharpoonup ~ : A \times B\rightarrow A$ that makes the tuple $(A, B, \vartriangleright, \vartriangleleft, \rightharpoonup, \leftharpoonup )$ into a matched pair of associative algebras. Explicitly, these bilinear maps are given by \begin{align*}
      %   (a,0)\Join (0,x) = (a \leftharpoonup x,~ a \vartriangleright x) ~~~ \text{ and } ~~~
       %  (0,x) \Join (a,0) = (x \rightharpoonup a,~ x \vartriangleleft a),
    % \end{align*}
% for~all $a\in A,x\in B.$
%Thus, a matched pair of associative algebras (or a twilled associative algebra) can be realized as an associative algebra structure on the direct sum $A\oplus B$ of two vector spaces for which both $A$ and $B$ are subalgebras.
%\end{exam}

%These are examples of quasi-twilled associative algebras which are in general not twilled associative algebras. Some of these examples are given below.

\begin{exam}\label{ex6}
 Let $A$ be an associative algebra. Then $(A\oplus A, \Box)$ is a quasi-twilled associative algebra, where \begin{align*}
      (a, a') ~\Box~(b,b') := (a \cdot_A b'+ a' \cdot_A b,~ \! a \cdot_A b + a' \cdot_A b' ),~~~~\mathrm{ for~all~} a,a',b,b'\in A.
  \end{align*}
\end{exam}

\begin{exam}\label{ex7}
 Let $A$ be an associative algebra and $(B,\vartriangleright, \vartriangleleft)$ be an $A$-bimodule. For any Hochschild $2$-cocycle $\theta\in Z^2_{\mathrm{Hoch}}(A, B)$, one can define an associative algebra structure on the direct sum $A\oplus B$ with the product \begin{align*}
      (a,x)\ltimes_\theta (b,y) := (a \cdot_A b,~ \! a \vartriangleright y + x\vartriangleleft b +\theta(a,b)),
  \end{align*}
for all $(a,x),(b,y)\in A\oplus B$.  Then $(A\oplus B,\ltimes_\theta)$ is a quasi-twilled associative algebra.
\end{exam}
\begin{exam}\label{ex8}
    This example is a particular case of the previous one. For any associative algebra $A$, the pair $(A \oplus A , \ltimes_{-\mu})$ is a quasi-twilled associative algebra, where \begin{align*}
      (a, a') \ltimes_{-\mu} (b,b') := (a \cdot_A b,~a \cdot_A b' + a'\cdot_A b -a\cdot_A b),~~~~\mathrm{ for~all~} a,a',b,b'\in A.
      \end{align*}
       \end{exam}

       \medskip
       
In the following, we give the Maurer-Cartan characterization of quasi-twilled associative algebras. To do so, we first recall the classical Gerstenhaber bracket on the space of all multilinear maps on a vector space \cite{gers-ring}.
Let $A$ be any vector space (not necessarily an associative algebra). For any $f\in \mathrm{Hom}(A^{\otimes m}, A)$ and $g\in \mathrm{Hom}(A^{\otimes n}, A)$, their {\em Gerstenhaber bracket} $[f,g]\in \mathrm{Hom}(A^{\otimes {m+n-1}}, A)$ is defined by \begin{align*}
    [f,g] := f\diamond g - (-1)^{(m-1)(n-1)} ~ \! g \diamond f,
\end{align*}
where \begin{align*}
(f\diamond g) (a_1,\ldots,a_{m+n-1}) = \displaystyle\sum_{i=1}^{m} (-1)^{(i-1)(n-1)} ~f\big(a_1,\ldots,a_{i-1},g(a_i,\ldots,a_{i+n-1}),a_{i+n},\ldots,a_{m+n-1}\big),     
\end{align*}
for $a_1,\ldots,a_{m+n-1} \in A$. Then the space $\oplus_{n=0}^\infty \mathrm{Hom} (A^{\otimes {n+1}}, A)$ endowed with the above Gerstenhaber bracket $[~,~]$ is a graded Lie algebra. Note that an associative algebra structure on the vector space $A$ is equivalent to having a Maurer-Cartan element of the graded Lie algebra $\big(\oplus_{n=0}^\infty \mathrm{Hom} (A^{\otimes {n+1}}, A),[~,~]\big)$.

Next, let $A$ and $B$ be two vector spaces. Let $\mathcal{A}^{k,l}$ be the direct sum of all possible $(k+l)$ tensor powers of $A$ and $B$ in which $A$ appears $k$ times and $B$ appears $l$ times. For example, \begin{align*} \mathcal{A}^{2,0} =~& A \otimes A,\quad \mathcal{A}^{1,1} = (A \otimes B) \oplus (B \otimes A), \quad \mathcal{A}^{0,2} = B \otimes B \quad \text{ and }\\  & \mathcal{A}^{2,1} = (A \otimes A \otimes B) \oplus (A \otimes B \otimes A) \oplus (B \otimes A \otimes A) ~~~~ \text{etc.}
    \end{align*}
    Note that there is an isomorphism $(A\oplus B)^{\otimes n+1} \cong \oplus_{ \substack{k+l=n+1\\ k,l \geq 0}} \mathcal{A}^{k,l}$, for any $n\geq 0$. Now, consider the graded Lie algebra 
    \begin{align*}
         \mathfrak{g}= \big(\oplus_{n=0}^\infty \mathrm{Hom} \big((A\oplus B)^{\otimes n+1}, A\oplus B\big),[~,~] \big)
         \end{align*}
         on the space of all multilinear maps on the vector space $A\oplus B$ with the Gerstenhaber bracket. Recall that \cite{uchino-t} a map $f\in \mathrm{Hom} \big((A\oplus B)^{\otimes n+1}, A\oplus B\big)$ is said to have {\em bidegree} $k\lvert l$ with $k+l=n$ and $k,l \geq -1$ if it satisfies \begin{align*}
        f(\mathcal{A}^{k+1,l})\subset A,\quad f(\mathcal{A}^{k,l+1})\subset B ~~~  \text{ and } ~~~ f = 0~~ \text{otherwise}.
    \end{align*} 
    We denote the set of all maps of bidegree $k\lvert l$ by the notation $C^{k\lvert l}$. Then we have \begin{align*}
     C^{k\lvert l} \cong  \mathrm{Hom}(\mathcal{A}^{k+1,l},A) \oplus \mathrm{Hom}(\mathcal{A}^{k,l+1},B).   
    \end{align*} 
    Moreover, \begin{align}\label{a3}
        \mathfrak{g}_n = \mathrm{Hom} \big((A\oplus B)^{\otimes n+1}, A\oplus B\big) \cong ~ C^{n+1\lvert -1} \oplus C^{n\lvert 0} \oplus \cdots \oplus C^{0\lvert n} \oplus C^{-1\lvert n+1}.
    \end{align}
    The following result has been proved in \cite{uchino-t}.
    \begin{lemma}\label{Lkl}
    Let $f\in C^{k_f\lvert l_f}$ and $g\in C^{k_g\lvert l_g}$. Then we have 
        $[f,g]\in C^{k_f+k_g\lvert l_f+l_g}$.
    \end{lemma}
    It follows from the above lemma that $\oplus_{n=0}^\infty C^{n\lvert 0}$ and $\oplus_{n=0}^\infty C^{0\lvert n}$ are both graded Lie subalgebras of $\mathfrak{g}$. We will now construct two new graded Lie subalgebras of $\mathfrak{g}$. For any $n\geq 0$, we define \begin{align*}
        \mathscr{M}_n:=~& C^{n\lvert 0} \oplus C^{n-1\lvert 1} \oplus \cdots \oplus C^{1\lvert n-1} \oplus C^{0\lvert n} \subset \mathfrak{g}_n,\\
    \mathscr{Q}_n:=~& C^{n+1\lvert -1} \oplus C^{n\lvert 0}\oplus C^{n-1\lvert 1} \oplus \cdots \oplus C^{1\lvert n-1} \oplus C^{0\lvert n} = C^{n+1 | -1} \oplus \mathscr{M}_n \subset \mathfrak{g}_n,\\
    \mathscr{R}_n:=~& C^{n\lvert 0}\oplus C^{n-1\lvert 1} \oplus \cdots \oplus C^{1\lvert n-1} \oplus C^{0\lvert n} \oplus C^{-1 \lvert n+1} = \mathscr{M}_n \oplus C^{ -1 | n+1} \subset \mathfrak{g}_n.
    \end{align*}    
 \begin{prop}\label{Pmq}
     With the above notations $\oplus_{n=0}^\infty\mathscr{M}_n$,~$\oplus_{n=0}^\infty\mathscr{Q}_n$ and $\oplus_{n=0}^\infty\mathscr{R}_n$ are all graded Lie subalgebras of $\mathfrak{g}$.
 \end{prop}
 
 \begin{proof}
     Let $f\in \mathscr{M}_m$ and $g\in \mathscr{M}_n$. If $f\in C^{i\lvert m-i}$ and $g\in C^{j\lvert n-j}$ (with $0\leq i \leq m$ and $0\leq j \leq n)$ then it follows from Lemma \ref{Lkl} that $[f,g] \in C^{i+j\lvert m+n-i-j} \subset \mathscr{M}_{m+n} $.
     
    For the second part, let $f\in \mathscr{Q}_m$ and $g\in \mathscr{Q}_n$. If $f\in C^{i\lvert m-i}$ and $g\in C^{j\lvert n-j}$ (with $0\leq i \leq m+1$ and $0\leq j \leq n+1)$ then from Lemma \ref{Lkl}, we get that $[f,g] =0$ when $i=m+1$ and $j=n+1$. Otherwise, we have  $[f,g]\in C^{i+j\lvert m+n-i-j}~ (\mathrm{with}~ 0\leq i+j\leq m+n+1) \subset \mathscr{Q}_{m+n}$. Finally, the last part is similar.
 \end{proof}
 
 Next, let $A$ and $B$ be two vector spaces. Suppose there are bilinear maps 
 \begin{align*} 
    &\mu: A \times A\rightarrow A, \qquad  \nu: B \times B\rightarrow B,\\
    \vartriangleright ~: A \times B\rightarrow B, \quad \vartriangleleft ~ : B \times A & \rightarrow B, \quad  \rightharpoonup ~: B \times A\rightarrow A, \quad \leftharpoonup ~ : A \times B\rightarrow A ~~~ \text{ and } ~~~  \theta: A \times A\rightarrow B.
    \end{align*}  
    Then we define elements $\widetilde{\theta}\in C^{2\lvert -1}$, ~$\mu_{_{\vartriangleright,\vartriangleleft}}\in C^{1\lvert 0}$ and $\nu_{_{\rightharpoonup,\leftharpoonup}} \in C^{0\lvert 1} $ by 
    \begin{align*}
     \widetilde{\theta}\big((a,x),(b,y)\big):=~& (0,\theta(a,b)),\\
\mu_{_{\vartriangleright,\vartriangleleft}}\big((a,x),(b,y)\big):=~& (a\cdot_A b, ~ \! a\vartriangleright y+x\vartriangleleft b),\\ \nu_{_{\rightharpoonup,\leftharpoonup}}\big((a,x),(b,y)\big):=~& (a\leftharpoonup y+x\rightharpoonup b, ~ \! x\cdot_B y),
    \end{align*}
    for $(a,x),(b,y)\in A\oplus B.$ Therefore, we obtain an element 
    \begin{align*}
      \Omega  = \widetilde{\theta} + \mu_{_{\vartriangleright,\vartriangleleft}} + \nu_{_{\rightharpoonup,\leftharpoonup}} \in C^{2\lvert -1} \oplus C^{1\lvert 0} \oplus C^{0\lvert 1} = \mathscr{Q}_1.
    \end{align*}
    \begin{thm}
        With the above notations, the tuple $(A_\mu, B_\nu, \vartriangleright, \vartriangleleft, \rightharpoonup, \leftharpoonup,\theta )$ is a quasi-twilled associative algebra if and only if the element $\Omega =\widetilde{\theta} + \mu_{_{\vartriangleright,\vartriangleleft}} + \nu_{_{\rightharpoonup,\leftharpoonup}} \in \mathscr{Q}_1$ is a Maurer-Cartan element of the graded Lie algebra $\mathscr{Q}=(\oplus_{n=0}^\infty\mathscr{Q}^n,[~,~])$.
    \end{thm} 
    
    \begin{proof}
        %Note that \begin{align*}
        %    \Omega \big((a,x),(b,y)\big) =~& \big(\widetilde{\theta}+\mu_{_{\vartriangleright,\vartriangleleft}}+\nu_{_{\rightharpoonup,\leftharpoonup}}\big)\big((a,x),(b,y)\big)\\
        %    =~&\big(a\cdot_A b+a\leftharpoonup y+x\rightharpoonup b,~ x\cdot_B y +a\vartriangleright y+x\vartriangleleft b+\theta(a,b)\big),
        %\end{align*}
        %for $(a,x),(b,y)\in A\oplus B$. Moreover, 
        We have \begin{align*}
            &[\Omega,\Omega] \\
            &= [\widetilde{\theta} + \mu_{_{\vartriangleright,\vartriangleleft}} + \nu_{_{\rightharpoonup,\leftharpoonup}} ~ \! , ~ \! \widetilde{\theta} + \mu_{_{\vartriangleright,\vartriangleleft}} + \nu_{_{\rightharpoonup,\leftharpoonup}} ]\\ 
            &= \underbrace{[\widetilde{\theta},\mu_{_{\vartriangleright,\vartriangleleft}}]+[\mu_{_{\vartriangleright,\vartriangleleft}},\widetilde{\theta}]}_{\in C^{3\lvert-1}} +  \underbrace{[\widetilde{\theta},~\nu_{_{\rightharpoonup,\leftharpoonup}}]+[\mu_{_{\vartriangleright,\vartriangleleft}},\mu_{_{\vartriangleright,\vartriangleleft}}]+[\nu_{_{\rightharpoonup,\leftharpoonup}},\widetilde{\theta}]}_{\in C^{2\lvert 0}} +  \underbrace{[\mu_{_{\vartriangleright,\vartriangleleft}},\nu_{_{\rightharpoonup,\leftharpoonup}}]+[\nu_{_{\rightharpoonup,\leftharpoonup}},\mu_{_{\vartriangleright,\vartriangleleft}}]}_{\in C^{1\lvert 1}} + \underbrace{[\nu_{_{\rightharpoonup,\leftharpoonup}},\nu_{_{\rightharpoonup,\leftharpoonup}}]}_{\in C^{0\lvert 2}}.
        \end{align*}
        It follows that $\Omega$ is a Maurer-Cartan element of the graded Lie algebra $\mathscr{Q}$ if and only if
        \begin{align*}[\mu_{_{\vartriangleright,\vartriangleleft}},\widetilde{\theta}]=0,\qquad [\mu_{_{\vartriangleright,\vartriangleleft}},\mu_{_{\vartriangleright,\vartriangleleft}}]+ 2[\nu_{_{\rightharpoonup,\leftharpoonup}},\widetilde{\theta}]=0, \qquad [\mu_{_{\vartriangleright,\vartriangleleft}},\nu_{_{\rightharpoonup,\leftharpoonup}}]=0 \quad \mathrm{and} \quad [\nu_{_{\rightharpoonup,\leftharpoonup}},\nu_{_{\rightharpoonup,\leftharpoonup}}]=0.
        \end{align*}
        It is easy to observe that 
\begin{align*}
&[\mu_{_{\vartriangleright,\vartriangleleft}},\widetilde{\theta}]=0 ~\text{ is equivalent to } (\ref{b11}),\\
&[\mu_{_{\vartriangleright,\vartriangleleft}},\mu_{_{\vartriangleright,\vartriangleleft}}]+ 2[\nu_{_{\rightharpoonup,\leftharpoonup}},\widetilde{\theta}]=0 ~ \text{ is equivalent to the identities } (\ref{b1})\text{-}(\ref{b4}),\\
&[\mu_{_{\vartriangleright,\vartriangleleft}},\nu_{_{\rightharpoonup,\leftharpoonup}}]= 0 ~ \text{ is equivalent to the identities } (\ref{b5})\text{-}(\ref{b10}),\\
&[\nu_{_{\rightharpoonup,\leftharpoonup}},\nu_{_{\rightharpoonup,\leftharpoonup}}]=0 ~\text{ is equivalent that } B_\nu \text{ is an associative algebra and } (A, \rightharpoonup, \leftharpoonup) \text{ is a } B_\nu\text{-bimodule}.
\end{align*}
   Hence $\Omega$ is a  Maurer-Cartan element of the graded Lie algebra $\mathscr{Q}$ if and only if $(A_\mu, B_\nu, \vartriangleright, \vartriangleleft, \rightharpoonup, \leftharpoonup,\theta )$ is a quasi-twilled associative algebra.    
    \end{proof}
    Using the above Maurer-Cartan characterization of a quasi-twilled associative algebra, we will define cohomology. Let $(\mathbb{A}, \cdot_\mathbb{A})$ be a quasi-twilled associative algebra. Suppose it is equivalent to the tuple $(A_\mu, B_\nu, \vartriangleright, \vartriangleleft, \rightharpoonup, \leftharpoonup,\theta )$. For each $n\geq 0$, we define the space of $n$-cochains $C^n( \mathbb{A} )$ by \begin{align*}
        C^0( \mathbb{A} ) :=~& A\oplus B,\\
    C^{n\geq 1}( \mathbb{A} ) :=~& \mathscr{Q}_{n-1} = C^{n\lvert -1} \oplus C^{n-1\lvert 0} \oplus \cdots \oplus C^{0\lvert n-1}.   
    \end{align*}
  Thus an element $F\in C^{n\geq 1}( \mathbb{A} )$ is given by an $(n+1)$ tuple $F= (F_0,F_1,\ldots,F_n)$, where $F_r \in C^{n-r\lvert r-1}$ for $r=0,1,\ldots,n.$ We also define a map $\delta:C^n( \mathbb{A} ) \rightarrow C^{n+1}( \mathbb{A} )$ by \begin{align*}
      \delta(F) = (-1)^{n-1} ~ \! [\widetilde{\theta} + \mu_{_{\vartriangleright,\vartriangleleft}} + \nu_{_{\rightharpoonup,\leftharpoonup}} ~ \!, F_0 + F_1 + \cdots + F_n],~~\mathrm{for}~~F=(F_0,F_1,\ldots,F_n) \in C^n (\mathbb{A}).
  \end{align*}
Explicitly, if $\delta (F) \in C^{n+1} (\mathbb{A})$ is represented by the tuple
$\big(  \delta (F)_0, \delta (F)_1, \ldots, \delta(F)_{n+1}  \big)$ then 
\begin{align*}
    \delta (F)_r = (-1)^{n-1} \big(   [\widetilde{ \theta}, F_{r+1}] + [ \mu_{_{\vartriangleright,\vartriangleleft}}, F_r] + [ \nu_{_{\rightharpoonup,\leftharpoonup}}, F_{r-1} ]  \big), \text{ for } r=0, 1, \ldots, n+1.
\end{align*}
Since the map $\delta$ is induced by the Maurer-Cartan element $\Omega = \widetilde{\theta} + \mu_{_{\vartriangleright,\vartriangleleft}} + \nu_{_{\rightharpoonup,\leftharpoonup}}$, we have $\delta^2 = 0$. The corresponding cohomology groups of the cochain complex $\{ C^\bullet (\mathbb{A}), \delta \} $ are called the {\bf cohomology groups} of the quasi-twilled associative algebra $(\mathbb{A}, \cdot_\mathbb{A})$.

%The cohomology groups of the cochain complex $\{C^{\bullet}(A_\mu, B_\nu, \vartriangleright, \vartriangleleft, \rightharpoonup, \leftharpoonup,\theta),\delta\}$ are called the cohomology groups of the quasi-twilled associative algebra $(A_\mu, B_\nu, \vartriangleright, \vartriangleleft, \rightharpoonup, \leftharpoonup,\theta)$. 

\begin{remark}
    In particular, if $(\mathbb{A}, \cdot_\mathbb{A})$ is a matched pair of associative algebras, we set 
    \begin{align*}
        C^0 (\mathbb{A}) := A \oplus B ~~~~ \text{ and } ~~~~ C^{n \geq 1} (\mathbb{A}) := \mathscr{M}_{n-1} = C^{n-1 | 0} \oplus \cdots \oplus C^{0 | n-1}.
    \end{align*}
    Hence an element $F \in C^{n \geq 1} (\mathbb{A})$ is an $n$-tuple $F = (F_1, \ldots, F_n)$, where $F_r \in C^{n-r\lvert r-1}$ for $r=1,\ldots,n.$ The differential $\delta : C^n (\mathbb{A}) \rightarrow C^{n+1} (\mathbb{A})$ is then defined by
    \begin{align*}
        \delta (F) = (-1)^{n-1} ~ \! [ \mu_{_{\vartriangleright,\vartriangleleft}} + \nu_{_{\rightharpoonup,\leftharpoonup}}, F_1 + \cdots + F_n ], \text{ for } F= (F_1, \ldots, F_n) \in C^n (\mathbb{A}).
    \end{align*}
    Componentwise, we have
    \begin{align*}
        \delta (F)_r = (-1)^{n-1} \big(  [ \mu_{_{\vartriangleright,\vartriangleleft}}, F_r] + [ \nu_{_{\rightharpoonup,\leftharpoonup}}, F_{r-1}]  \big), \text{ for } r=1, \ldots, n+1.
    \end{align*}
    In this case, the corresponding cohomology groups are called the {\bf cohomology groups} of the given matched pair of associative algebras.
\end{remark}

\medskip

\section{Strong deformation maps in quasi-twilled associative algrbras}\label{sec4}

In this section, we introduce strong deformation maps in a quasi-twilled associative algebra generalizing some well-known operators like algebra homomorphisms, derivations, crossed homomorphisms and modified associative {\sf r}-matrices. Then we introduce the cohomology of a strong deformation map unifying the existing cohomologies of all the operators mentioned above.

\medskip

Throughout this section, we let $(\mathbb{A}, \cdot_\mathbb{A})$ be a quasi-twilled associative algebra, where $\mathbb{A}$ has a fixed decomposition $\mathbb{A} = A \oplus B$ with $B \subset \mathbb{A}$ a subalgebra. In terms of components, we assume that this quasi-twilled associative algebra is equivalent to the tuple $(A_\mu, B_\nu, \vartriangleright, \vartriangleleft,  \rightharpoonup, \leftharpoonup,\theta)$ (cf. Theorem \ref{explicit-qtaa}). Therefore, the associative multiplication $\cdot_\mathbb{A}$ is simply given by (\ref{omega-a-map}).

\medskip

 \begin{defn}\label{defn-strong-dm} A \textbf{strong deformation map} in the quasi-twilled associative algebra $(\mathbb{A}, \cdot_\mathbb{A})$ is a linear map $D:A\rightarrow B$ that satisfies
 \begin{align}\label{a15}
         D\big(a\cdot_A b + a\leftharpoonup D(b) + D(a)\rightharpoonup b  \big) = D(a)\cdot_B D(b) + a\vartriangleright D(b)+ D(a)\vartriangleleft b  +\theta(a,b),~~\mathrm{for~ all~} a,b\in A. 
    \end{align}
    \end{defn}
    
    %\textcolor{red}{A pair $\big((A, B, \vartriangleright, \vartriangleleft, \rightharpoonup, \leftharpoonup,\theta), D\big)$ consisting of a quasi-twilled associative algebra and a strong deformation map on it is called a quasi-twilled strong deformation algebra.}

\begin{exam}
    Let $A$ and $B$ be two associative algebras. Consider the quasi-twilled associative algebra $(A\oplus B,\odot)$ in Example \ref{ex1}. Then a linear map $D: A \rightarrow B$ is a strong deformation map in $(A\oplus B,\odot)$ if and only if it is an {\bf algebra homomorphism}, i.e.
    \begin{align*}
   D(a\cdot_A b) = D(a)\cdot_B D(b) ,~~\mathrm{for~all~} a,b \in A.
 \end{align*}
\end{exam}
\begin{exam}
    Let $A$ be an associative algebra and $(B, \vartriangleright, \vartriangleleft)$ be an $A$-bimodule. We consider the quasi-twilled associative algebra $(A\oplus B,\ltimes)$ given in Example \ref{ex2}. In this case, a linear map $D: A \rightarrow B$ is a strong deformation map in $(A\oplus B,\ltimes)$ if and only if it is a {\bf derivation} on $A$ with values in the $A$-bimodule $(B, \vartriangleright, \vartriangleleft)$, i.e. \begin{align*}
     D(a\cdot_A b) =  D(a)\vartriangleleft b 
 + a\vartriangleright D(b),~~\mathrm{for~all~} a,b \in A. 
    \end{align*}
\end{exam}
\begin{exam}
    Let $A$ be an associative algebra and $B$ be an associative algebra object in the category of $A$-bimodules. A linear map $D: A \rightarrow B$ is a strong deformation map in the quasi-twilled associative algebra $(A\oplus B,\odot_\ltimes)$ given in Example \ref{ex3} if and only if $D$ is a {\bf crossed homomorphism}, i.e. 
    \begin{align*}
    D(a\cdot_A b) =  D(a)\cdot_B D(b)+ D(a)\vartriangleleft b + a\vartriangleright D(b)  ,~~\mathrm{for~all~} a,b \in A.     
    \end{align*}
\end{exam}
\begin{exam}
    Let $A$ be an associative algebra and consider the quasi-twilled associative algebra $(A\oplus A, \Box)$ given in Example \ref{ex6}. Then a linear map $D: A \rightarrow A$ is a strong deformation map in $(A\oplus A,\Box)$ if and only if $D$ is a {\bf modified associative {\sf r}-matrix}, i.e.
    \begin{align*}
      D(a)\cdot_A D(b) = D\big(D(a)\cdot_A b + a\cdot_A D(b)\big) - a\cdot_A b, ~~\mathrm{for~all~} a,b \in A. 
    \end{align*}
\end{exam}

   \begin{remark}
     Let $A$ be an associative algebra and $(B,\vartriangleright, \vartriangleleft)$ be an $A$-bimodule such that the second Hochschild cohomology group $H^2_{\mathrm{Hoch}}(A, B)$ is not trivial. Suppose $\theta\in Z^2_{\mathrm{Hoch}}(A,B)$ is a Hochschild $2$-cocycle representing a nonzero element in $H^2_{\mathrm{Hoch}}(A,B)$, i.e. $\theta$ is not an exact $2$-cocycle. We now consider the quasi-twilled associative algebra $(A\oplus B,\ltimes_\theta)$ given in Example \ref{ex7}. Note that a linear map $D: A \rightarrow B$ is a strong deformation map if and only if 
     \begin{align*} 
  D(a\cdot_A b ) = ~D(a)\vartriangleleft b + a\vartriangleright D(b) + \theta(a,b),~~\mathrm{for~all~} a,b \in A.
     \end{align*}
     This holds if and only if $\theta =- (\delta_{\mathrm{Hoch}}D)$ is an exact $2$-cocycle which is not possible. Therefore, a strong deformation map may not exist in arbitrary quasi-twilled associative algebra.
    \end{remark}

\medskip

In the following, we characterize strong deformation maps in terms of their graphs.
\begin{prop}
Let $(\mathbb{A}, \cdot_\mathbb{A})$ be a quasi-twilled associative algebra. Then a linear map $D: A\rightarrow B$ is a strong deformation map if and only if the graph $\mathrm{Gr}(D) = \{ (a, D(a)) \lvert ~ a\in A \}$ is a subalgebra of the associative algebra $(\mathbb{A}, \cdot_\mathbb{A})$.   
\end{prop}
\begin{proof}
    Let $(a,D(a))$ and $(b,D(b))$ be two arbitary elements in $\mathrm{Gr}(D)$. Then we have \begin{align*}
    (a,D(a)) \cdot_\mathbb{A} (b,D(b)) = \big(a\cdot_A b + a\leftharpoonup D(b) + D(a)\rightharpoonup b ,~ D(a)\cdot_B D(b)+a\vartriangleright D(b) + D(a)\vartriangleleft b +\theta(a,b)\big).  
    \end{align*} 
     This is in $\mathrm{Gr}(D)$ if and only if the identity (\ref{a15}) holds.
\end{proof}
As an application of the previous proposition, we get the following.
\begin{prop}
    Let $D: A\rightarrow B$ be a strong deformation map in the quasi-twilled associative algebra $(\mathbb{A}, \cdot_\mathbb{A}).$
    %$(A_\mu, B_\nu, \vartriangleright, \vartriangleleft, \rightharpoonup, \leftharpoonup,\theta)$.
    \begin{itemize}
        \item [(i)] Then the vector space $A$ endowed with the multiplication
        \begin{align}\label{a16}
           \mu_D : A \times A \rightarrow A, ~ (a, b) \mapsto a \cdot_D b :=a\cdot_A b +  a\leftharpoonup D(b) +  D(a)\rightharpoonup b,~~\mathrm{for~} a,b \in A
        \end{align}
        is an associative algebra, denoted by $A_D$.
        \item[(ii)] Moreover, the tuple $(A_{\mu_D}, B_\nu, \vartriangleright_D, \vartriangleleft_D, \rightharpoonup, \leftharpoonup)$ is a matched pair of associative algebras, where for $a\in A$ and $x\in B$, \begin{align}\label{a17}
         a\vartriangleright_D x :=&~a\vartriangleright x + D(a)\cdot_B x - D(a\leftharpoonup x),\\ \label{a18}
         x\vartriangleleft_D a :=&~x\vartriangleleft a + x\cdot_B D(a) - D(x\rightharpoonup a).
        \end{align} 
    \end{itemize}
\end{prop}
\begin{proof}
    (i) Since $\mathrm{Gr}(D)$ is linearly isomorphic to the vector space $A$, it follows from the previous proposition that the map $D$ induces an associative algebra structure on $A$. The induced associative multiplication on $A$ is given by (\ref{a16}).
    
    (ii) Let $(a,x)\in \mathbb{A}$ be any arbitary element. Then we can write \begin{align*}
     (a,x) = (a,D(a))+ (0,x-D(a))\in \mathrm{Gr}(D) \oplus B.
    \end{align*}
    Therefore, we have $\mathbb{A} = \mathrm{Gr}(D) \oplus B$. On the other hand, $D$ is a strong deformation map implying that $\mathrm{Gr}(D)$ is a subalgebra of $(\mathbb{A}, \cdot_\mathbb{A})$. This shows that the pair $(\mathrm{Gr}(D), B )$ forms a matched pair of associative algebras. Under the isomorphism $\mathrm{Gr}(D) \cong A$, $(a,D(a)) \longleftrightarrow a$, we get that $(A_{\mu_D}, B_\nu, \vartriangleright_D, \vartriangleleft_D, \rightharpoonup, \leftharpoonup)$ is a matched pair of associative algebras, where $\vartriangleright_D$ and $\vartriangleleft_D$ are respectively given by (\ref{a17}) and (\ref{a18}).
    \end{proof}

    It follows from the previous proposition that $A_D$ is an associative algebra and $(B,\vartriangleright_D,\vartriangleleft_D)$ is an $A_D$-bimodule. We define the Hochschild cohomology groups of $A_D$ with coefficients in the $A_D$-bimodule $(B,\vartriangleright_D,\vartriangleleft_D)$ as the {cohomology groups} of the strong deformation map $D$. Explicitly, we let  \begin{align*}
    C^n(D):= \mathrm{Hom}(A^{\otimes n},B),~~\mathrm{for~}n\geq 0
\end{align*}
and define a map $\delta^D: C^n(D) \rightarrow C^{n+1}(D)$ by 
\begin{align*}
&\delta^D(f)(a_1,\ldots,a_{n+1})\\
&= a_1\vartriangleright_D f(a_2,\ldots,a_{n+1}) + \displaystyle\sum_{i=1}^{n} (-1)^{i}  f(a_1,\ldots,a_{i-1},a_i\cdot_D a_{i+1},\ldots,a_{n+1}) + (-1)^{n+1} f(a_1,\ldots,a_n) \vartriangleleft_D a_{n+1},\\
&= a_1\vartriangleright f(a_2,\ldots,a_{n+1}) + D(a_1)\cdot_B f(a_2,\ldots,a_{n+1})  - D\big(a_1\leftharpoonup f(a_2,\ldots,a_{n+1})\big)\\
& \quad + \displaystyle\sum_{i=1}^{n} (-1)^{i}  f\big(a_1,\ldots,a_{i-1},a_i\cdot_A a_{i+1} + a_i\leftharpoonup D(a_{i+1}) +  D(a_i)\rightharpoonup a_{i+1} , \ldots,a_{n+1}\big)\\
& \quad + (-1)^{n+1} \big\{  f(a_1,\ldots,a_n) \vartriangleleft a_{n+1} + f(a_1,\ldots,a_n) \cdot_B D(a_{n+1}) - D\big( f(a_1,\ldots,a_n)\rightharpoonup a_{n+1} \big) \big\},
\end{align*}
for $f\in C^n(D) = \mathrm{Hom}(A^{\otimes n},B)$ and $a_1,\ldots,a_{n+1}\in A$. Then $(\delta^D)^2 = 0$. The cohomology groups of the cochain complex $\{C^{\bullet}(D),\delta^D \}$ are called the {\bf cohomology groups} of the strong deformation map $D$. The corresponding cohomology groups are denoted by $H^{\bullet}(D)$. 

\medskip

The cohomologies for algebra homomorphisms \cite{gers-sch}, derivations \cite{loday-morph}, crossed homomorphisms \cite{das-crossed} and modified associative $\mathsf{r}$-matrices \cite{das-modified} are well studied in the literature. The cohomologies for all these operators are defined by Hochschild cohomology of certain associative algebras (induced by the respective operators) with coefficients in a suitable induced bimodule. Here we list the induced associative algebras and corresponding bimodules for each of the above operators.

\medskip

\medskip

%\begin{center}
%\begin{tabular}{ c| c| c }
% cell1 & cell2 & cell3 \\ \hline
% cell4 & cell5 & cell6 \\  
% cell7 & cell8 & cell9    
%\end{tabular}
%\end{center}

%\begin{table}[htbp]
%    \centering
   % \resizebox{\textwidth}{!}{
%    \begin{tabular}{|c|c|c|}
%    \hline
%    Operator $D$ & Induced associative structure & Induced bimodule\\
 %    & on the domain of $D$ & on the codomain of $D$ \\
 %   \hline
 %   algebra homomorphism & $ a \cdot_D b = a \cdot_A b$ & $a \vartriangleright_D x = D(a) \cdot_B x,$ \\
 %   $D: A \rightarrow B$ & & $x \vartriangleleft_D a = x \cdot_B D(a),$  \\
 %   \hline derivation
 % & $a \cdot_D b =  a \cdot_A b$ & $a \vartriangleright_D x = a \vartriangleright x,$ \\
 %    $D: A \rightarrow B$ & & $x \vartriangleleft_D a = x \vartriangleleft a,$  \\
%    \hline
%crossed homomorphism & $a \cdot_D b =  a \cdot_A b$ & $a \vartriangleright_D x = a \vartriangleright x + D(a) \cdot_B x,$ \\
%   $D: A \rightarrow B$ & & $x \vartriangleleft_D a = x \vartriangleleft a + x \cdot_B D(a),$  \\
 %   \hline modified associative $\mathsf{r}$-matrix & $a \cdot_D b = D(a) \cdot_A b + a \cdot_A D(b) $ & $a \vartriangleright_D b = D(a) \cdot_A b - D (a \cdot_A b),$ \\
%      $D: A \rightarrow A$  & & $b \vartriangleleft_D a = b \cdot_A D(a) - D (b \cdot_A a).$ \\
%    \hline
%    \end{tabular}
    %\caption{Caption}
   % \label{tab:my_label}
%\end{table}

\begin{center}
    \begin{tabular}{|c|c|c|}
    \hline
    Operator $D$ & Induced associative structure & Induced bimodule\\
     & on the domain of $D$ & on the codomain of $D$ \\
    \hline
    algebra homomorphism & $ a \cdot_D b = a \cdot_A b$ & $a \vartriangleright_D x = D(a) \cdot_B x,$ \\
    $D: A \rightarrow B$ & & $x \vartriangleleft_D a = x \cdot_B D(a),$  \\
    \hline derivation
  & $a \cdot_D b =  a \cdot_A b$ & $a \vartriangleright_D x = a \vartriangleright x,$ \\
     $D: A \rightarrow B$ & & $x \vartriangleleft_D a = x \vartriangleleft a,$  \\
    \hline
crossed homomorphism & $a \cdot_D b =  a \cdot_A b$ & $a \vartriangleright_D x = a \vartriangleright x + D(a) \cdot_B x,$ \\
   $D: A \rightarrow B$ & & $x \vartriangleleft_D a = x \vartriangleleft a + x \cdot_B D(a),$  \\
    \hline modified associative $\mathsf{r}$-matrix & $a \cdot_D b = D(a) \cdot_A b + a \cdot_A D(b) $ & $a \vartriangleright_D b = D(a) \cdot_A b - D (a \cdot_A b),$ \\
      $D: A \rightarrow A$  & & $b \vartriangleleft_D a = b \cdot_A D(a) - D (b \cdot_A a).$ \\
    \hline
    \end{tabular}
    \end{center}

\medskip

\medskip

The above table shows that our cohomology of a strong deformation map in a quasi-twilled associative algebra unifies the cohomologies for all the operators mentioned above.

\medskip

\section{Controlling algebras for strong deformation maps and simultaneous deformations}\label{sec5}

Given a quasi-twilled associative algebra $\mathbb{A}$, we first construct a curved $L_\infty$-algebra whose Maurer-Cartan elements are precisely strong deformation maps in $\mathbb{A}$. This curved $L_\infty$-algebra is the controlling algebra for strong deformation maps in $\mathbb{A}$. Next, for a (fixed) strong deformation map $D$, we obtain the $L_\infty$-algebra that governs the linear deformations of the operator $D$ keeping the underlying quasi-twilled associative algebra $\mathbb{A}$ intact. Finally, we also construct the governing $L_\infty$-algebra for studying the simultaneous deformations of both $\mathbb{A}$ and $D$.

\subsection{Controlling algebras for strong deformation maps} Let $(\mathbb{A}, \cdot_\mathbb{A})$ be a quasi-twilled associative algebra that is equivalent to the tuple $(A_\mu, B_\nu, \vartriangleright, \vartriangleleft, \rightharpoonup, \leftharpoonup, \theta )$. We first consider the graded Lie algebra with the Gerstenhaber bracket
\begin{align*}
    \mathfrak{g}= \big(\oplus_{n=0}^\infty \mathrm{Hom} \big((A\oplus B)^{\otimes n+1}, A\oplus B \big),[~,~]\big)
\end{align*}
on the space of all multilinear maps on the vector space $A \oplus B$. Then it is easy to see that the graded subspace $\mathfrak{a}= \oplus_{n=0}^\infty \mathrm{Hom} (A^{\otimes n+1}, B)$ is an abelian graded Lie subalgebra of $\mathfrak{g}$. Let $P:\mathfrak{g}\rightarrow \mathfrak{a}$ be the projection map onto the subspace $\mathfrak{a}$. Then $\mathrm{ker}(P)_n= C^{n\lvert 0} \oplus C^{n-1\lvert 1} \oplus \cdots \oplus C^{1 \lvert n-1} \oplus C^{0\lvert n} \oplus C^{-1\lvert n+1}$, for any $n\geq 0$. Hence $\mathrm{ker}(P)=\oplus_{n=0}^\infty \mathrm{ker}(P)_n$ is a graded Lie subalgebra of $\mathfrak{g}$ (by Proposition \ref{Pmq}).
Let \begin{align*}
\Omega=\widetilde{\theta}+\mu_{_{\vartriangleright,\vartriangleleft}}+\nu_{_{\rightharpoonup,\leftharpoonup}}
\end{align*}
be the Maurer-Cartan element corresponding to the given quasi-twilled associative algebra. Note that $P(\Omega)=\theta$. Therefore, we obtain a curved $V$-data $(\mathfrak{g},\mathfrak{a}, P,\Omega)$. As a result, we get the following curved $L_{\infty}$-algebra. 
\begin{thm}\label{Tqtalcla}
Let $(A_\mu, B_\nu, \vartriangleright, \vartriangleleft, \rightharpoonup, \leftharpoonup, \theta )$ be a quasi-twilled associative algebra. Then the pair 
\begin{align*}
\big( \mathfrak{a} = \oplus_{n=0}^\infty\mathrm{Hom} (A^{\otimes n+1}, B),\{l_k\}_{k=0}^{\infty}\big)
\end{align*}
is a curved $L_{\infty}$-algebra, where for $f, g \in \mathfrak{a}$, \begin{align*}
l_0=~&\theta,\\
     l_1(f)=~& [\mu_{_{\vartriangleright,\vartriangleleft}},f],\\
  l_2(f,g)=~&[[\nu_{_{\rightharpoonup,\leftharpoonup}},f],g],\\
 l_k=~&0,~~\mathrm{for~}k\geq 3.
 \end{align*}
 Moreover, a linear map $D: A\rightarrow B$ is a strong deformation map in the quasi-twilled associative algebra $(A_\mu, B_\nu, \vartriangleright, \vartriangleleft, \rightharpoonup, \leftharpoonup, \theta )$ if and only if $D\in\mathfrak{a}_0 = \mathrm{Hom}(A, B)$ is a Maurer-Cartan element of the above-defined curved $L_{\infty}$-algebra $(\mathfrak{a},\{l_k\}_{k=0}^{\infty})$.
\end{thm}
\begin{proof}
    Since $(\mathfrak{g},\mathfrak{a}, P,\Omega)$ is a curved $V$-data, it follows from Theorem \ref{Tcvd}(i) that $(\mathfrak{a},\{l_k\}_{k=0}^{\infty})$ is a curved $L_{\infty}$-algebra, where \begin{align*} l_0=~& P(\Omega) =\theta,\\ 
    l_1(f)=~& P[\Omega,f]= [\mu_{_{\vartriangleright,\vartriangleleft}},f],\\
    .l_2(f,g)=~&P[[\Omega,f],g]=[[\nu_{_{\rightharpoonup,\leftharpoonup}},f],g].    
  \end{align*}
  Moreover, using the bidegree reason, one can easily check that $[ [ \Omega, f], g] \in \mathfrak{a}$, for any $f, g \in \mathfrak{a}$. Hence for any (other) $h \in \mathfrak{a}$, we have 
  \begin{align}\label{three-br}
  [[  [ \Omega, f], g], h] = 0 
  \end{align}
  as $\mathfrak{a}$ is an abelian Lie subalgebra. This shows that
  $      l_k(f_1, f_2, \ldots,f_k) = P[\cdots[[\Omega,f_1],f_2],\ldots,f_k] = 0,$ for any $k \geq 3$ and $f_1,f_2,\ldots,f_k \in \mathfrak{a}$. This proves the first part.
  
  Next, for any linear map $D:A\rightarrow B$, we have 
  \begin{align*}
      &\big(l_0+\displaystyle\sum_{k=1}^{\infty} \dfrac{1}{k!}~l_k(D,\ldots,D)\big) (a,b)\\
      &= \big(\theta+l_1(D)+\frac{1}{2}l_2(D,D)\big)(a,b)\\
      &= \theta(a,b) + [\mu_{_{\vartriangleright,\vartriangleleft}},D](a,b)+ \frac{1}{2}[[\nu_{_{\rightharpoonup,\leftharpoonup}},D],D](a,b)\\
      &= \theta(a,b)+a\vartriangleright D(b)+ D(a)\vartriangleleft b - D(a\cdot_A b) +D(a)\cdot_B D(b)- D\big(a\leftharpoonup D(b)\big) - D\big(D(a)\rightharpoonup b\big),
  \end{align*}
  for all $a,b\in A$. This shows that $l_0+\displaystyle\sum_{k=1}^{\infty} \dfrac{1}{k!}~l_k(D,\ldots,D) =0$ if and only if $D$ is a strong deformation map.
\end{proof}

The curved $L_\infty$-algebra $\big( \mathfrak{a} = \oplus_{n=0}^\infty\mathrm{Hom} (A^{\otimes n+1}, B),\{l_k\}_{k=0}^{\infty}\big)$ is called the {\bf controlling algebra} for strong deformation maps. It is important to remark that this controlling algebra generalizes the controlling algebras for algebra homomorphisms \cite{gers-sch}, derivations \cite{loday-morph} and crossed homomorphism \cite{das-crossed}. As a new application, if we consider the quasi-twilled associative algebra $(A \oplus A, \Box)$ given in Example \ref{ex6}, we obtain the controlling algebra for modified associative $r$-matrices.

\begin{prop}
    Let $A$ be an associative algebra. Then $( \oplus_{n=0}^\infty \mathrm{Hom} (A^{\otimes n+1}, A) , \{ l_k \}_{k=0}^\infty)$ is a curved $L_\infty$-algebra, where $l_0 = \mu$ (the algebra multiplication on $A$), $l_k = 0$ for $k =1, 3, 4, \ldots$ and
    \begin{align*}
        l_2 (f, g)& (a_1, a_2, \ldots, a_{m+n}) = (-1)^m \sum_{i=1}^m (-1)^{(i-1)n} ~ \! f \big(   a_1, \ldots, a_{i-1}, g(a_i, \ldots, a_{i+n-1}) \cdot a_{i+n}, \ldots, a_{m+n}   \big) \\
        &- (-1)^m \sum_{i=1}^m (-1)^{in} f \big(  a_1, \ldots, a_{i-1}, a_i \cdot g (a_{i+1}, \ldots, a_{i+n}), a_{i+n+1}, \ldots, a_{m+n}   \big) \\
        &- (-1)^{ m (n+1)} \bigg\{ \sum_{i=1}^n (-1)^{(i-1)m} ~ \! g \big(   a_1, \ldots, a_{i-1}, f(a_i, \ldots, a_{i+m-1}) \cdot a_{i+m}, \ldots, a_{m+n}   \big) \\
        & \qquad \qquad \qquad \quad  - \sum_{i=1}^n (-1)^{im} g \big(  a_1, \ldots, a_{i-1}, a_i \cdot f (a_{i+1}, \ldots, a_{i+m}), a_{i+m+1}, \ldots, a_{m+n}   \big) \bigg\} \\
        &+ (-1)^{m (n+1)} f(a_1, \ldots, a_m ) \cdot g (a_{m+1}, \ldots, a_{m+n}) - (-1)^m ~ \! g(a_1, \ldots, a_n) \cdot f (a_{n+1}, \ldots, a_{m+n}),
    \end{align*}
    for $f \in \mathrm{Hom}(A^{\otimes m}, A)$ and $g \in \mathrm{Hom}(A^{\otimes n}, A)$. Moreover, a linear map $D: A \rightarrow A$ is a modified associative $r$-matrix on $A$ if and only if $D$ is a Maurer-Cartan element of this curved $L_\infty$-algebra.
\end{prop}

%We will now use the above Maurer-Cartan characterization of a strong deformation map $D$ to obtain a new $L_\infty$-algebra.
\begin{thm}\label{thm-govern-d}
    Let $(A_\mu, B_\nu, \vartriangleright, \vartriangleleft, \rightharpoonup, \leftharpoonup,\theta)$ be a quasi-twilled associative algebra and $D: A\rightarrow B$ be a fixed strong deformation map. Then $\big(\mathfrak{a} = \oplus_{n=0}^\infty\mathrm{Hom} (A^{\otimes n+1}, B),\{l_k^D\}_{k=1}^{\infty}\big)$ is an $L_{\infty}$-algebra, where for $f, g \in \mathfrak{a},$ 
    \begin{align}\label{a19}
        l_1^D (f) =~& l_1(f) + l_2(D,f),\\ \label{a20}
     l_2^D(f,g) =~& l_2(f,g),\\ \label{a21}
     l_k^D =~& 0,~~\mathrm{for~} k\geq 3.
    \end{align}
    Moreover, for any linear map $D':A\rightarrow B$, the sum $D+D'$ is also a strong deformation map if and only if $D' \in \mathfrak{a}_0$ is a Maurer-Cartan element of the $L_{\infty}$-algebra $(\mathfrak{a},\{l_k^D\}_{k=1}^{\infty})$ given above.
\end{thm}
\begin{proof}
    Since $D: A\rightarrow B$ is a strong deformation map, it follows from the previous theorem that $D$ is a Maurer-Cartan element of the curved $L_{\infty}$-algebra $(\mathfrak{a},\{l_k\}_{k=0}^{\infty})$. Hence by applying Theorem \ref{Tmcla} in the present context, we get that $(\mathfrak{a},\{l_k^D\}_{k=1}^{\infty})$ is an $L_{\infty}$-algebra, where the structure maps $\{l_k^D\}_{k=1}^{\infty}$ are given by (\ref{a19})-(\ref{a21}).
    
    For the second part, we observe that \begin{align*}
     &l_0+  l_1(D+D') + \frac{1}{2} l_2(D+D',D+D')\\
     &=\underbrace{l_0+l_1(D) + \frac{1}{2} l_2(D,D)}_{=0} + \big\{l_1(D') +  l_2(D,D')+ \frac{1}{2} l_2(D',D')\big\}\\
     &= l_1^D(D') + \frac{1}{2} l_2^D(D',D').
    \end{align*}
    This shows that $D+D'$ is a Maurer-Cartan element of the curved $L_{\infty}$-algebra $(\mathfrak{a},\{l_k \}_{k=0}^{\infty})$ if and only if $D'$ is a Maurer-Cartan element of the $L_{\infty}$-algebra $(\mathfrak{a},\{l_k^D\}_{k=1}^{\infty})$. Hence the result follows.
\end{proof}

\begin{remark}
Since the $L_\infty$-algebra $\big(\mathfrak{a} = \oplus_{n=0}^\infty\mathrm{Hom} (A^{\otimes n+1}, B),\{l_k^D\}_{k=1}^{\infty}\big)$ constructed in the preceding theorem governs the linear deformations of $D$, we call it the {\bf governing algebra} of the map $D$.
\end{remark}

Let $D: A\rightarrow B$ be a strong deformation map in the quasi-twilled associative algebra. Then we have seen in Theorem \ref{thm-govern-d} that $(\mathfrak{a} = \oplus_{n=0}^\infty\mathrm{Hom} (A^{\otimes n+1}, B),\{l_k^D\}_{k=1}^{\infty})$ is an $L_{\infty}$-algebra. As a consequence, we get that the degree $1$ map $l_1^D: \mathfrak{a}\rightarrow \mathfrak{a}$ is a differential. Moreover, we have the following.
\begin{prop}
     For any $f\in \mathfrak{a}_{n-1} = \mathrm{Hom}(A^{\otimes n},B)$, we have \begin{align*}
         l_1^D(f) = (-1)^{n-1} ~ \! \delta^D(f).
     \end{align*}   
    \end{prop}
    \begin{proof}
        For any $(a,x),(b,y)\in \mathbb{A}$, we have \begin{align*}
            [\nu_{_{\rightharpoonup,\leftharpoonup}},D]((a,x),(b,y))=~ \big( a\leftharpoonup D(b) + D(a)\rightharpoonup b, D(a)\cdot_B y+x\cdot_B D(b)-D(a\leftharpoonup y)-D(x\rightharpoonup b)\big).
        \end{align*}
        This shows that $\mu_{_{\vartriangleright,\vartriangleleft}}+[\nu_{_{\rightharpoonup,\leftharpoonup}},D]=~\mu_D+\vartriangleright_D+\vartriangleleft_D$ (using the notations from (\ref{a16}), (\ref{a17}), (\ref{a18})). Hence, we have  \begin{align*}
         l_1^D (f) =~& l_1(f) + l_2(D,f)\\
         =~& [\mu_{_{\vartriangleright,\vartriangleleft}},f]+[[\nu_{_{\rightharpoonup,\leftharpoonup}},D],f]\\
         =~& [\mu_{_{\vartriangleright,\vartriangleleft}}+[\nu_{_{\rightharpoonup,\leftharpoonup}},D],f]\\
         =~&[\mu_D+\vartriangleright_D+\vartriangleleft_D,f]\\
        =~& (-1)^{n-1} ~ \! \delta^D(f). 
        \end{align*}
        Hence the proof.
    \end{proof}
    \subsection{Simultaneous deformations of a quasi-twilled associative algebra and a strong deformation map} 
Here we first construct an $L_\infty$-algebra whose Maurer-Cartan elements correspond to pairs $(\mathbb{A}, D)$, where $\mathbb{A} = A \oplus B$ has a quasi-twilled associative algebra structure and $D: A \rightarrow B$ is a strong deformation map in $\mathbb{A}$. Then we study the governing algebra and deformations of such a pair $(\mathbb{A}, D)$.
    
    Let $A$ and $B$ be two vector spaces (not necessarily equipped with additional structures). Let $\mathfrak{g}= \big(\oplus_{n=0}^\infty \mathrm{Hom} \big((A\oplus B)^{\otimes n+1}, A\oplus B \big),[~,~]\big)$ be the Gerstenhaber graded Lie algebra on the space of all multilinear maps on $A\oplus B$. Then we have seen earlier that $\mathfrak{a}= \oplus_{n=0}^\infty \mathrm{Hom} (A^{\otimes n+1}, B)$ is an abelian graded Lie subalgebra of $\mathfrak{g}$. Moreover, if $P:\mathfrak{g}\rightarrow \mathfrak{a}$ is the projection map onto the subspace $\mathfrak{a}$ then $\mathrm{ker}(P)\subset \mathfrak{g}$ is a graded Lie subalgebra. Hence the quadruple $(\mathfrak{g},\mathfrak{a}, P,\Delta=0)$ is a $V$-data. Thus by applying Theorem \ref{Tcvd} (ii) to the graded Lie subalgebra $\mathscr{Q} \subset \mathfrak{g}$, we get the following result.
    \begin{thm}
    Let $A$ and $B$ be two vector spaces. Then $(\mathscr{Q}[1]\oplus \mathfrak{a},\{l_k\}_{k=1}^{\infty})$ is an $L_{\infty}$-algebra, where \begin{align*} 
\widetilde{l}_1( F[1],f) =~& P(F),\\ 
\widetilde{l}_2(F[1], G[1]) =~& (-1)^{\lvert F \rvert} [F,G][1] ,\\
\widetilde{l}_k\big(F[1],f_1,\ldots,f_{k-1}\big) =~& P[\cdots [[F,f_1],f_2],\ldots,f_{k-1} ],~~k\geq 2,
\end{align*}
for all homogeneous elements $F, G \in \mathscr{Q}$ and $f,f_1,\ldots,f_{k-1}\in \mathfrak{a}$.
    \end{thm}
    Next, suppose there are bilinear maps  
    \begin{align*} 
    &\mu: A \times A\rightarrow A, \qquad  \nu: B \times B\rightarrow B,\\
    \vartriangleright ~: A \times B\rightarrow B, \quad \vartriangleleft ~ : B \times A & \rightarrow B, \quad  \rightharpoonup ~: B \times A\rightarrow A, \quad \leftharpoonup ~ : A \times B\rightarrow A , \quad \theta: A \times A\rightarrow B
    \end{align*}
    and a linear map $D:A \rightarrow B$. We now consider the element 
    \begin{align*}
      \Omega :=\widetilde{\theta}+\mu_{_{\vartriangleright,\vartriangleleft}}+\nu_{_{\rightharpoonup,\leftharpoonup}}\in C^{2\lvert -1} \oplus C^{1\lvert 0} \oplus C^{0\lvert 1} = \mathscr{Q}_1.
    \end{align*}
     The same element $\Omega\in \mathscr{Q}_1$ can be realized as an element $\Omega[1]\in (\mathscr{Q}[1])_0$ by degree shift.

    \begin{thm}\label{mc-simulaneous}
        With the above notations, $(A_\mu, B_\nu, \vartriangleright, \vartriangleleft, \rightharpoonup, \leftharpoonup,\theta )$ is a quasi-twilled associative algebra and $D: A \rightarrow B$ is a strong deformation map in it if and only if the element $\alpha =(\Omega[1], D)\in (\mathscr{Q}[1]\oplus \mathfrak{a})_0$ is a Maurer-Cartan element of the $L_{\infty}$-algebra $(\mathscr{Q}[1]\oplus \mathfrak{a},\{ \widetilde{l}_k\}_{k=1}^{\infty})$.
    \end{thm}
    \begin{proof}
First, by using an argument similar to (\ref{three-br}), one can show that
 \begin{align*}
            [[[\Omega,D],D],D] =~0.
        \end{align*}
        As a result, we get that $\widetilde{l}_k\big((\Omega[1],D), \ldots,(\Omega[1],D) \big)=~0,$ for all $k\geq 4$. Hence 
        \begin{align*}
            &\displaystyle\sum_{k=1}^{\infty} \dfrac{1}{k!}~\widetilde{l}_k\big((\Omega[1],D),\ldots,(\Omega[1],D)\big)\\
            &= \widetilde{l}_1(\Omega[1],D) + \frac{1}{2} ~ \! \widetilde{l}_2((\Omega[1],D),(\Omega[1],D))+ \frac{1}{3!} ~ \! \widetilde{l}_3((\Omega[1],D),(\Omega[1],D),(\Omega[1],D)) \\
            &= (0,P(\Omega))+\frac{1}{2}\big\{\widetilde{l}_1(\Omega[1],\Omega[1])+ 2 ~ \! \widetilde{l}_2(\Omega[1],D)\big\}+ \frac{1}{3!}\big\{3 ~ \! \widetilde{l}_3(\Omega[1],D,D)\big\}\\
            &= \big(-\frac{1}{2} [\Omega,\Omega][1], P(\Omega)+P[\Omega,D]+ \frac{1}{2} [[\Omega,D],D]\big)\\
            &= \big(-\frac{1}{2} [\Omega,\Omega][1],~ \! \theta+l_1(D)+\frac{1}{2}l_2(D,D)\big).
        \end{align*}
        We have already seen that $[\Omega,\Omega]=0$ if and only if $(A_\mu, B_\nu, \vartriangleright, \vartriangleleft, \rightharpoonup, \leftharpoonup, \theta )$ is a quasi-twilled associative algebra. Moreover, $\theta+l_1(D)+\frac{1}{2} l_2(D, D)=0$ if and only if $D: A \rightarrow B$ is a strong deformation map in the above quasi-twilled associative algebra (cf. Theorem \ref{Tqtalcla}). This proves the result.
    \end{proof}

    Let $(\mathbb{A}, D)$ be a pair of a quasi-twilled associative algebra $\mathbb{A} = ( \mathbb{A}, \cdot_\mathbb{A})$ with a fixed decomposition $\mathbb{A} = A \oplus B$ and a strong deformation map $D$ in it. Then we have seen in Theorem \ref{mc-simulaneous} that $\alpha = ( \Omega [1], D)$ is a Maurer-Cartan element of the $L_\infty$-algebra $(\mathscr{Q}[1]\oplus \mathfrak{a},\{ \widetilde{l}_k\}_{k=1}^{\infty})$, where $\Omega \in \mathscr{Q}_1$ represents the given quasi-twilled associative algebra structure on $\mathbb{A}$. Hence, we obtain the following by applying Theorem \ref{Tmcla}.

    \begin{thm}
          Let $(\mathbb{A}, D)$ be a pair of a quasi-twilled associative algebra $\mathbb{A}$ with a fixed decomposition $\mathbb{A} = A \oplus B$ and a strong deformation map $D$ in it. Then the pair $  \big(  \mathscr{Q}[1]\oplus \mathfrak{a},\{ \widetilde{l}_k^{   ( \Omega [1], D)  } \}_{k=1}^{\infty}  \big)$ is an $L_\infty$-algebra, where
          \begin{align*}
              \widetilde{l}_k^{   ( \Omega [1], D)  }  \big(  (F_1 [1], f_1), \ldots, (F_k [1], f_k) \big) := \sum_{n=0}^\infty \frac{1}{n!} ~ \! \widetilde{l}_{n+k} \big(  \underbrace{  ( \Omega [1], D), \ldots,    ( \Omega [1], D)}_{n \mathrm{~ times}},  (F_1 [1], f_1), \ldots, (F_k [1], f_k)  \big),  
          \end{align*}
          for $k \geq 1$ and homogeneous elements $ (F_1 [1], f_1), \ldots, (F_k [1], f_k) \in  \mathscr{Q}[1]\oplus \mathfrak{a}$. Moreover, for any other element $\Omega' \in \mathscr{Q}_1$ and a linear map $D': A \rightarrow B$, the sum $\Omega + \Omega' \in \mathscr{Q}_1$ represents a new quasi-twilled associative algebra structure on $\mathbb{A}$ and $D + D': A \rightarrow B$ is a strong deformation map in this quasi-twilled associative structure if and only if $(\Omega' [1], D') \in (\mathscr{Q}[1]\oplus \mathfrak{a})_0$ is a Maurer-Cartan element of  $  \big(  \mathscr{Q}[1]\oplus \mathfrak{a},\{ \widetilde{l}_k^{   ( \Omega [1], D)  } \}_{k=1}^{\infty}  \big)$.
    \end{thm}

    The above theorem says that the $L_\infty$-algebra  $  \big(  \mathscr{Q}[1]\oplus \mathfrak{a},\{ \widetilde{l}_k^{   ( \Omega [1], D)  } \}_{k=1}^{\infty}  \big)$ governs the simultaneous deformations of the given quasi-twilled associative algebra structure on $\mathbb{A}$ and the strong deformation map $D$. For this reason, we call this $L_\infty$-algebra as the {\bf governing algebra} for the pair $(\mathbb{A}, D)$.

\section{Weak deformation maps in quasi-twilled associative algrbras}\label{sec6}
%The notion of deformation maps in a matched pair of associative algebras (or a twilled associative algebra) was introduced by Agore \cite{agore} in the study of classifying compliments. 
This section introduces weak deformation maps in a quasi-twilled associative algebra. We observe that such maps unify various operators like (relative) Rota-Baxter operators of weight $0$, (relative) Rota-Baxter operators of weight $1$, twisted Rota-Baxter operators, Reynolds operators, left averaging operators and right averaging operators. Then we define the cohomology of a weak deformation map generalizing the existing cohomologies of all the operators mentioned above.

Let $(\mathbb{A}, \cdot_\mathbb{A})$ be a quasi-twilled associative algebra. We assume that this quasi-twilled associative algebra is equivalent to the tuple $(A_\mu, B_\nu, \vartriangleright, \vartriangleleft, \rightharpoonup, \leftharpoonup,\theta)$.

\begin{defn}
 A \textbf{weak deformation map} in the quasi-twilled associative algebra $(\mathbb{A}, \cdot_\mathbb{A})$ is a linear map $r:B\rightarrow A$ that satisfies 
 \begin{align}\label{a4}
        r(x)\cdot_A r(y) + r(x)\leftharpoonup y + x\rightharpoonup r(y) =~ r\big(x\cdot_B y + r(x)\vartriangleright y + x\vartriangleleft r(y) +\theta(r(x),r(y))\big),~~\mathrm{for~ all~} x,y\in B.
    \end{align}
\end{defn}

Let $r : B \rightarrow A$ be a weak deformation map. If $r$ is invertible then $r^{-1} : A \rightarrow B$ is a strong deformation map in the sense of Definition \ref{defn-strong-dm}. Conversely, the inverse (when invertible) of a strong deformation map is a weak deformation map.
    
%\textcolor{red}{      A pair $(\mathbb{A},r)$ consisting of a quasi-twilled associative algebra $\mathbb{A}$ and a deformation map $r$ in $\mathbb{A}$ is called a \textbf{quasi-twilled deformation algebra}.}

\begin{exam}
 Let $A$ and $B$ be two associative algebras. Consider the quasi-twilled associative algebra $(A\oplus B,\odot)$ given in Example \ref{ex1}. A linear map $r : B \rightarrow A$ is a weak deformation map in $(A \oplus B, \odot)$ if and only if $r$ is an {\bf algebra homomorphism}, i.e.
 \begin{align*}
   r(x\cdot_B y) = r(x) \cdot_A r(y),~~\mathrm{for~all~} x,y \in B.
 \end{align*}
\end{exam}

\begin{exam}
 Let $A$ be an associative algebra and $(B, \vartriangleright, \vartriangleleft)$ be an $A$-bimodule. Consider the quasi-twilled associative algebra $(A\oplus B,\ltimes)$ given in Example \ref{ex2}. A linear map $r: B\rightarrow A$ is a weak deformation map in $(A\oplus B,\ltimes)$ if and only if $r$ is a {\bf relative Rota-Baxter operator of weight $0$}, i.e.
 \begin{align*}
     r(x)\cdot_A r(y) = r\big(r(x) \vartriangleright y + x\vartriangleleft r(y)\big),~~\mathrm{for~all~} x,y \in B.
     \end{align*}
\end{exam}

\begin{exam}
Let $A$ be an associative algebra and $B$ be an associative algebra object in the category of $A$-bimodules. Consider the quasi-twilled associative algebra $(A\oplus B,\odot_\ltimes)$ given in Example \ref{ex3}. In this case, a linear map $r : B \rightarrow A$ is a weak deformation map in $(A\oplus B,\odot_\ltimes)$ if and only if $r$ is a {\bf relative Rota-Baxter operator of weight $1$}, i.e.
\begin{align*} 
r(x)\cdot_A r(y) = r\big(x\cdot_B y+r(x) \vartriangleright y + x\vartriangleleft r(y)\big),~~\mathrm{for~all~} x,y \in B.
     \end{align*}
\end{exam}

\begin{exam}
 Let $A$ be an associative algebra and $(B, \vartriangleright)$ be a left $A$-module. Here we consider the quasi-twilled associative algebra $(A\oplus B,\cdot_\vartriangleright)$ given in Example \ref{ex4}. Then a linear map $r: B\rightarrow A$ is a weak deformation map in $(A\oplus B,\cdot_\vartriangleright)$
if and only if $r$ is a {\bf relative left averaging operator}, i.e.
\begin{align*}
     r(x)\cdot_A r(y) = r\big(r(x) \vartriangleright y \big),~~\mathrm{for~all~} x,y \in B.
     \end{align*} 
\end{exam}

\begin{exam}
  Let $A$ be an associative algebra and $(B, \vartriangleleft)$ be a right $A$-module. Next, we consider the quasi-twilled associative algebra $(A\oplus B,\cdot_\vartriangleleft)$ given in Example \ref{ex5}. A linear map $r: B\rightarrow A$ is a weak deformation map in $(A\oplus B,\cdot_\vartriangleleft)$ if and only if $r$ is a {\bf relative right averaging operator}, i.e.
 \begin{align*}
     r(x)\cdot_A r(y) = r\big(x \vartriangleleft r(y) \big),~~\mathrm{for~all~} x,y \in B.
     \end{align*}  
\end{exam}

\begin{exam}
 Let $A$ be an associative algebra, $(B,\vartriangleright, \vartriangleleft)$ be an $A$-bimodule and $\theta\in Z^2_{\mathrm{Hoch}}(A,B)$ be a Hochschild 2-cocycle.  A linear map $r : B \rightarrow A$ is a weak deformation map in the quasi-twilled associative algebra $(A\oplus B,\ltimes_\theta)$ given in Example \ref{ex7} if and only if $r$ is a {\bf $\theta$-twisted Rota-Baxter operator}, i.e.
 \begin{align*} 
r(x)\cdot_A r(y) = r\big(r(x) \vartriangleright y + x\vartriangleleft r(y)+\theta(r(x),r(y) ) \big),~~\mathrm{for~all~} x,y \in B.
     \end{align*}
\end{exam}

\begin{exam}
 Let $A$ be an associative algebra. Consider the quasi-twilled  associative algebra $(A\oplus A,\ltimes_{-\mu})$ given in Example \ref{ex8}. A linear map $r: A\rightarrow A$ is a weak deformation map in $(A\oplus A,\ltimes_{-\mu})$ if and only if $r$ is a {\bf Reynolds operator}, i.e.
 \begin{align*} 
r(a)\cdot_A r(b) = r\big(r(a) \cdot_A b + a\cdot_A r(b)-r(a)\cdot_A r(b)\big),~~\mathrm{for~all~} a,b \in A.
     \end{align*}
\end{exam}

\begin{exam}
Let $(A, B, \vartriangleright, \vartriangleleft, \rightharpoonup, \leftharpoonup)$ be a matched pair of associative algebras and consider the quasi-twilled associative algebra $(A \oplus B, \Join)$ given in Example \ref{matched-quasi}. A linear map $r: B \rightarrow A$ is a weak deformation map in $(A \oplus B, \Join)$ if and only if $r$ satisfies
\begin{align}\label{agore-defor}
     r(x)\cdot_A r(y) + r(x)\leftharpoonup y + x\rightharpoonup r(y) =~ r\big(x\cdot_B y + r(x)\vartriangleright y + x\vartriangleleft r(y) \big),~~\mathrm{for~ all~} x,y\in B.
\end{align}
In other words, $r$ is a deformation map \cite{agore} in the given matched pair of associative algebras.
\end{exam}

%\textcolor{red}{Last section}\\
%By replacing the Hochschild 2-cocycle with any non-abelian 2-cocycle, one obtains the following.
%\begin{exam}
% Let $A$ and $B$ be two associative algebras and $(\vartriangleright, \vartriangleleft,\theta)$ be a non-abelian 2-cocycle on $A$ with values in $B$. Consider the quasi-twilled associative algebra $(A\oplus B,\boxtimes)$ described in Example \ref{ex9}. Thus, a deformation map in $(A\oplus B,\boxtimes)$ is a linear map $r: B\rightarrow A$ satisfying \begin{align*} 
%r(x)\cdot_A r(y) = r\big( x\cdot_B y+ r(x) \vartriangleright y + x\vartriangleleft r(y)+\theta(r(x),r(y)\big),~~\mathrm{for~all~} x,y \in B.
 %    \end{align*}
  %    We call such a map by $(\vartriangleright, \vartriangleleft,\theta)$-twisted Rota-Baxter operator.     
%\end{exam}

The next result gives a characterization of weak deformation maps in a quasi-twilled associative algebra.

\begin{prop}
    Let $(\mathbb{A}, \cdot_\mathbb{A})$ be a quasi-twilled associative algebra. A linear map $r: B\rightarrow A$ is a weak deformation map if and only if its graph $\mathrm{Gr}(r) = \{ (r(x),x)\lvert ~ x\in B \}$ is a subalgebra of the associative algebra $(\mathbb{A}, \cdot_\mathbb{A})$.
\end{prop}

\begin{proof}
    Let $(r(x),x)$ and $(r(y),y)$ be two arbitary elements in $\mathrm{Gr}(r)$. Then it follows from (\ref{omega-a-map}) that 
    \begin{align*}
      (r(x),x) \cdot_\mathbb{A} (r(y),y) = \big(r(x)\cdot_A r(y) + r(x)\leftharpoonup y + x\rightharpoonup r(y),~x\cdot_B y+r(x) \vartriangleright y + x\vartriangleleft r(y)+\theta\big(r(x),r(y)\big)\big). 
    \end{align*} 
  This is in $\mathrm{Gr}(r)$ if and only if the identity (\ref{a4}) holds. This proves the desired result.   
\end{proof}

As a consequence of the above proposition, we get the following.
\begin{prop}\label{r-def}
Let $(\mathbb{A}, \cdot_\mathbb{A} )$ be a quasi-twilled associative algebra and $r: B\rightarrow A$ be a weak deformation map in it. Then the vector space $B$ is equipped with the new multiplication
\begin{align*}
 \nu_r : B \times B \rightarrow B, ~(x, y) \mapsto   x\cdot_r y := x\cdot_B y+r(x) \vartriangleright y + x\vartriangleleft r(y)+\theta(r(x),r(y)),~~\mathrm{for~} x,y\in B
\end{align*}
is an associative algebra, denoted by $B_r$.
\end{prop}

In the next, we prove a more general result. More precisely, we show that a weak deformation map in a quasi-twilled associative algebra $(\mathbb{A}, \cdot_\mathbb{A})$ deforms the structure into a new quasi-twilled associative algebra. In terms of components, we have the following.

\begin{thm}
Let $(A_\mu, B_\nu, \vartriangleright, \vartriangleleft, \rightharpoonup, \leftharpoonup, \theta )$ be a quasi-twilled associative algebra and $r: B\rightarrow A$ be a weak deformation map in it. Then $(A_{\mu_r}, B_{\nu_r}, \vartriangleright_r, \vartriangleleft_r, \rightharpoonup_r, \leftharpoonup_r, \theta)$ is a quasi-twilled associative algebra, where for any $a,b\in A$ and $x,y\in B$,
\begin{align} \label{a5}
   \mu_r (a, b) = a\cdot_r b:=~& a\cdot_A b - r(\theta(a,b)),\\\label{a6}
   \nu_r (x, y) = x\cdot_r y:=~& x\cdot_B y +r(x) \vartriangleright y + x\vartriangleleft r(y)+\theta(r(x),r(y)),\\\label{a7}
  a\vartriangleright_r x:=~& a\vartriangleright x + \theta (a,r(x)),\\ \label{a8}
  x\vartriangleleft_r a:=~& x\vartriangleleft a + \theta(r(x),a),\\ \label{a9}
  x\rightharpoonup_ra:=~& x\rightharpoonup a + r(x)\cdot_A a- r(x \vartriangleleft a) - r\big(\theta(r(x),a)\big),\\ \label{a10}
 a\leftharpoonup_r x:=~& a\leftharpoonup x + a\cdot_A r(x) - r(a \vartriangleright x) - r\big(\theta(a,r(x))\big). 
\end{align}
\end{thm}
\begin{proof}
    Let $\Omega = \widetilde{\theta} + \mu_{_{ \vartriangleright, \vartriangleleft}} + \nu_{_{\rightharpoonup, \leftharpoonup}} \in C^{2 | -1} \oplus C^{1 | 0} \oplus C^{0| 1} = \mathscr{Q}_1$ be the Maurer-Cartan element corresponding to the given quasi-twilled associative algebra. Then for any linear map $r: B \rightarrow A$, one can define a new element
    \begin{align*}
        \Omega_r := \widetilde{\theta} + (\mu_{_{ \vartriangleright, \vartriangleleft}})_r + (\nu_{_{\rightharpoonup, \leftharpoonup}})_r + \psi_r ~ \in C^{2 | -1} \oplus C^{1 | 0} \oplus C^{0| 1}  \oplus C^{-1 | 2} = \mathfrak{g}_1,
    \end{align*}
    where
    \begin{align*}
        (\mu_{_{ \vartriangleright, \vartriangleleft}})_r :=~& \mu_{_{ \vartriangleright, \vartriangleleft}} + [ \widetilde{\theta}, r], \qquad (\nu_{_{\rightharpoonup, \leftharpoonup}})_r := \nu_{_{\rightharpoonup, \leftharpoonup}} + [ \mu_{_{ \vartriangleright, \vartriangleleft}}, r] + \frac{1}{2} [[ \widetilde{\theta}, r ], r] \\
        &\text{ and } ~~~\psi_r := [\nu_{_{\rightharpoonup, \leftharpoonup}}, r] + \frac{1}{2!} [[ \mu_{_{ \vartriangleright, \vartriangleleft}}, r ], r] + \frac{1}{3!} [[[ \widetilde{\theta}, r], r], r].
    \end{align*}
By straightforward calculations, we see that
    \begin{align}
        ( \mu_{_{ \vartriangleright, \vartriangleleft}} )_r ((a, x), (b, y)) =~& (a \cdot_r b, ~ \!  a \vartriangleright_r y + x \vartriangleleft_r b), \label{r-mu}\\
        (\nu_{_{\rightharpoonup, \leftharpoonup}})_r  ((a, x), (b, y)) =~& (a \leftharpoonup_r y + x \rightharpoonup_r b, ~ \! x \cdot_r y), \label{r-nu}\\
        \psi_r  ((a, x), (b, y)) =~& \big( \substack{ -r(x\cdot_B y)+ r(x)\leftharpoonup y + x\rightharpoonup r(y)  \\ ~~  + r(x)\cdot_A r(y)- r ( r(x)\vartriangleright y + x\vartriangleleft r(y) ) - r(\theta (r(x),r(y) ))}
        ~ \! , 0 \big), \label{r-psi}
    \end{align}
    for $(a, x), (b, y) \in A \oplus B$. It has been shown by Uchino \cite{uchino-t} that $[\Omega_r, \Omega_r ] = 0$, i.e. the element $\Omega_r$ defined a new associative algebra structure on the direct sum $A \oplus B$.  However, if $r$ is a weak deformation map then $\psi_r  = 0$. As a result, $\Omega_r =   \widetilde{\theta} + (\mu_{_{ \vartriangleright, \vartriangleleft}})_r + (\nu_{_{\rightharpoonup, \leftharpoonup}})_r \in \mathscr{Q}_1$ corresponds to a quasi-twilled associative algebra structure. Finally, it follows from (\ref{r-mu}) and (\ref{r-nu}) that this structure is precisely $(A_{\mu_r}, B_{\nu_r}, \vartriangleright_r, \vartriangleleft_r, \rightharpoonup_r, \leftharpoonup_r, \theta)$ in terms of components.
\end{proof}

%In the following, we define the cohomology of a weak deformation map in a quasi-twilled associative algebra. Our cohomology is made in such a way that it unifies the existing cohomologies of relative Rota-Baxter operators, twisted Rota-Baxter operators, Reynolds operators, left-averaging operators and right-averaging operators.

%Let $(A, B, \vartriangleright, \vartriangleleft, \rightharpoonup, \leftharpoonup, \theta )$ be a quasi-twilled associative algebra and $r: B\rightarrow A$ be a weak deformation map. Since $(A_{\mu_r}, B_{\nu_r}, \vartriangleright_r, \vartriangleleft_r, \rightharpoonup_r, \leftharpoonup_r, \theta )$ is a quasi-twilled associative algebra, it follows that $B_{\nu_r}$ is an associative algebra and $(A,\rightharpoonup_r, \leftharpoonup_r)$ is a bimodule over it. 

Let $(\mathbb{A}, \cdot_\mathbb{A})$ be a quasi-twilled associative algebra and $r : B \rightarrow A$ be a weak deformation map in $(\mathbb{A}, \cdot_\mathbb{A})$. We have already seen in Proposition \ref{r-def} that $B_r$ is an associative algebra. In particular, the above theorem shows that the triple $(A, \rightharpoonup_r, \leftharpoonup_r)$ is a $B_r$-bimodule. We define the
Hochschild cohomology groups of $B_r$ with coefficients in the $B_r$-bimodule $(A, \rightharpoonup_r, \leftharpoonup_r)$ as the cohomology groups of the weak deformation map $r$.
More precisely, we set
\begin{align*}
    C^n(r):= \mathrm{Hom}(B^{\otimes n},A),~~\mathrm{for~}n\geq 0 
    \end{align*}
and a map $\delta^r: C^n(r) \rightarrow C^{n+1}(r)$ by 
\begin{align}
&\delta^r(f)(x_1,\ldots,x_{n+1}) \label{hoch-r-weak}\\ 
&= x_1 \rightharpoonup_r f(x_2,\ldots,x_{n+1}) + \displaystyle\sum_{i=1}^{n} (-1)^{i} ~ f(x_1,\ldots,x_{i-1},x_i\cdot_r x_{i+1},\ldots,x_{n+1})  \nonumber \\
& \qquad \qquad \qquad  \qquad  + (-1)^{n+1} f(x_1,\ldots,x_n) \leftharpoonup_r x_{n+1}, \nonumber \\
&= x_1\rightharpoonup f(x_2,\ldots,x_{n+1}) + r(x_1)\cdot_A f(x_2,\ldots,x_{n+1})  \nonumber \\ & \qquad \qquad \qquad  \qquad  - r\big(x_1\vartriangleleft f(x_2,\ldots,x_{n+1})\big) - r\big(\theta\big(r(x_1), f(x_2,\ldots,x_{n+1})\big)\big) \nonumber  \\ 
& \quad + \displaystyle \sum_{i=1}^{n} (-1)^{i} ~ f\big(x_1,\ldots,x_{i-1},x_i\cdot_B x_{i+1}+  r(x_i)\vartriangleright x_{i+1}+ x_i\vartriangleleft r(x_{i+1})+\theta\big(r(x_i),r(x_{i+1})\big),\ldots,x_{n+1}\big)  \nonumber \\ 
& \quad + (-1)^{n+1} \big\{  f(x_1,\ldots,x_n) \leftharpoonup x_{n+1} + f(x_1,\ldots,x_n) \cdot_A r(x_{n+1})  \nonumber \\ 
& \qquad \qquad \qquad  \qquad  - r(f(x_1,\ldots,x_n)\vartriangleright x_{n+1})- r\big(\theta\big(f(x_1,\ldots,x_n),r(x_{n+1})\big)\big)\big\}, \nonumber 
\end{align}
for $f\in C^n(r) = \mathrm{Hom} (B^{\otimes n}, A)$ and $x_1,\ldots,x_{n+1}\in B$. Then $(\delta^r)^2 = 0$. The cohomology groups of the cochain complex $\{ C^\bullet (r), \delta^r \}$ are called the {\bf cohomology groups} of the weak deformation map $r$. We denote the corresponding cohomology groups by $H^\bullet(r)$. 

Recently, the cohomologies for relative Rota-Baxter operators (of weight $0$) \cite{das-rota}, relative Rota-Baxter operators of weight $1$ \cite{das-weighted}, twisted Rota-Baxter operators \cite{das-ns}, Reynolds operators \cite{das-ns}, left averaging operators \cite{wang-zhou} and right averaging operators \cite{wang-zhou} are extensively studied. The cohomologies for all these operators are defined by Hochschild cohomology of some associative algebras (induced by the respective operators) with coefficients in a suitable induced bimodule. Here we list the induced associative algebras and corresponding bimodules for each of the above operators.

\medskip

\medskip

\begin{center}
    \begin{tabular}{|c|c|c|}
    \hline
    Operator $r$ & Induced associative structure & Induced bimodule\\
     & on the domain of $r$ &  on the codomain of $r$\\
    \hline
    relative Rota-Baxter&  & $x \rightharpoonup_r a = r(x) \cdot_A a - r (x \vartriangleleft a),$ \\
     operator of weight $0$  & $x\cdot_r y=~r(x) \vartriangleright y+ x\vartriangleleft r(y)$ & $a \leftharpoonup_r x= a \cdot_A r (x) - r (a \vartriangleright x),$\\
     $r: B \rightarrow A$ & & \\
    \hline relative Rota-Baxter 
  &  & $x \rightharpoonup_r a = r(x) \cdot_A a - r (x \vartriangleleft a), $\\
  operator of weight $1$ & $x\cdot_r y=~ x\cdot_B y + r(x) \vartriangleright y + x\vartriangleleft r(y)$ & $a \leftharpoonup_r x= a \cdot_A r (x) - r (a \vartriangleright x),$\\
  $r: B \rightarrow A$ & & \\
    \hline
$\theta$-twisted  &  & $x \rightharpoonup_r a = r(x) \cdot_A a - r (x \vartriangleleft a)$\\
Rota-Baxter operator &  $x\cdot_r y=~r(x) \vartriangleright y+ x\vartriangleleft r(y) + \theta(r(x),r(y))$ & \qquad \qquad   $ - ~ r \big(  \theta ( r (x) , a)\big),$ \\
$r: B \rightarrow A$ & & $a \leftharpoonup_r x= a \cdot_A r (x) - r (a \vartriangleright x)$  \\
 & & \qquad \qquad $ - ~r \big( \theta (a, r (x)) \big),$\\
    \hline &    & $a \rightharpoonup_r b = r(a) \cdot_A b - r (a \cdot_A b)$\\
     Reynolds operator & $a\cdot_r b=~r(a) \cdot_A b+ a\cdot_A r(b) - r(a) \cdot_A r(b) $ & \qquad \qquad $ +~ r (r(a) \cdot_A b),$\\
    $r : A \rightarrow A$ & & $b \leftharpoonup_r a= b \cdot_A r (a) - r (b \cdot_A a)$ \\ 
     & & \qquad \qquad $+ r (b \cdot_A r (a)),$\\
    \hline
    relative left averaging&  & $x \rightharpoonup_r a = r (x) \cdot_A a,$ \\
     operator  & $x\cdot_r y=~r(x) \vartriangleright y$ & $a \leftharpoonup_r x= a \cdot_A r (x) - r (a \vartriangleright x),$ \\
    $r: B \rightarrow A$ & & \\
    \hline
    relative right averaging &  & $x \rightharpoonup_r a = r(x) \cdot_A a - r (x \vartriangleleft a),$ \\
     operator  & $x\cdot_r y=~x\vartriangleleft r(y) $ & $a \leftharpoonup_r x= a \cdot_A r (x)$ \\
    $r: B \rightarrow A$ & & \\
    \hline
    \end{tabular}
\end{center}

\medskip

\medskip

It follows from the table that our cohomology of a weak deformation map in a quasi-twilled associative algebra unifies the cohomologies for all the operators mentioned above.

\medskip

\section{Controlling algebras for weak deformation maps and simultaneous deformations}\label{sec7}

In this section, we first construct an $L_\infty$-algebra whose Maurer-Cartan elements correspond to weak deformation maps in a given quasi-twilled associative algebra. For a (fixed) weak deformation map $r$, we also construct an $L_\infty$-algebra that governs the linear deformations of the operator $r$. In the end, we define the governing $L_\infty$-algebra of a pair $(\mathbb{A}, r)$ consisting of a given quasi-twilled associative algebra $\mathbb{A}$ and a weak deformation map $r$, and discuss simultaneous deformations of $\mathbb{A}$ and $r$.

\subsection{Controlling algebras for weak deformation maps}
Let $(\mathbb{A}, \cdot_\mathbb{A})$ be a quasi-twilled associative algebra. Suppose it is equivalent to the tuple $(A_\mu, B_\nu, \vartriangleright, \vartriangleleft, \rightharpoonup, \leftharpoonup, \theta )$. As before, we consider the graded Lie algebra 
\begin{align*}
    \mathfrak{g}= \big(\oplus_{n=0}^\infty \big(\mathrm{Hom} (A\oplus B)^{\otimes n+1}, A\oplus B\big),[~,~]\big).
\end{align*}
We define a graded subspace $\mathfrak{b} = \oplus_{n=0}^\infty \mathrm{Hom} (B^{\otimes n+1}, A)\subset \mathfrak{g}$. Then it is easy to see that $\mathfrak{b}$ is an abelian graded Lie subalgebra of $\mathfrak{g}$. Let $P:\mathfrak{g}\rightarrow \mathfrak{b}$ be the projection map onto the subspace $\mathfrak{b}$. Note that \begin{align*}
    \mathrm{ker }(P)_n= C^{n+1\lvert -1} \oplus C^{n\lvert 0} \oplus C^{n -1 \lvert 1} \oplus \cdots \oplus C^{1\lvert n-1} \oplus  C^{0\lvert n}  \quad (\text{in terms of the notation } (\ref{a3})).
\end{align*} 
Thus, it follows from Proposition \ref{Pmq} that $\mathrm{ker}(P) = \oplus_{n=0}^\infty \mathrm{ker} (P)_n \subset \mathfrak{g}$ is a graded Lie subalgebra. Let
\begin{align*}
\Omega=~\widetilde{\theta} + \mu_{_{\vartriangleright,\vartriangleleft}}+\nu_{_{\rightharpoonup,\leftharpoonup}}\in \mathfrak{g}_1
\end{align*}
be the Maurer-Cartan element corresponding to the given quasi-twilled associative algebra. Since $P (\Omega) = 0$, we get that the quadruple $(\mathfrak{g},\mathfrak{b}, P,\Omega)$ is a $V$-data. This leads to the following result. 

\begin{thm}\label{Tqtallar}
 Let $(A_\mu, B_\nu, \vartriangleright, \vartriangleleft, \rightharpoonup, \leftharpoonup, \theta )$ be a quasi-twilled associative algebra. Then the pair 
 \begin{align*}
 \big( \mathfrak{b} = \oplus_{n=0}^\infty\mathrm{Hom} (B^{\otimes n+1}, A),\{l_k\}_{k=1}^{\infty}\big)
 \end{align*}
 is an $L_{\infty}$-algebra, where for $f, g, h \in \mathfrak{b}$,
 \begin{align}
    l_1(f)=~& [\nu_{_{\rightharpoonup,\leftharpoonup}},f], \label{r-l1}\\
  l_2(f,g)=~&[[\mu_{_{\vartriangleright,\vartriangleleft}},f],g], \label{r-l2}\\
 l_3(f,g,h)=~& [[[\widetilde{\theta},f],g],h], \label{r-l3}\\
 l_k=~&0,~~\mathrm{for~}k\geq 4.
 \end{align} 
 Moreover, a linear map $r: B\rightarrow A$ is a weak deformation map if and only if $r \in \mathfrak{b}_0 = \mathrm{Hom} (B, A)$ is a Maurer-Cartan element of the above $L_{\infty}$-algebra.
\end{thm}

\begin{proof}
    Since $(\mathfrak{g},\mathfrak{b}, P,\Omega)$ is a $V$-data, it follows from Theorem \ref{Tcvd} (i) that $(\mathfrak{b},\{l_k\}_{k=1}^{\infty})$ is an $L_{\infty}$-algebra, where \begin{align*}
  l_k(f_1,f_2,\ldots,f_k) = P[\cdots[[\Omega,f_1],f_2],\ldots,f_k], \text{ for } k\geq 1.    
  \end{align*}
  %Since $\Omega\in C^{2\lvert -1} \oplus C^{1\lvert 0} \oplus C^{0\lvert 1}$ has no component in $\mathfrak{b}_1 = \mathrm{Hom} (B^{\otimes 2}, A) = C^{-1\lvert 2} $, we have $P(\Omega)=0$. Which implies that $l_0=0.$ Therefore, $(\mathfrak{b},\{l_k^\alpha\}_{k=1}^{\infty})$ is an $L_{\infty}$-algebra. 
  Note that, for any homogeneous elements $f , g, h \in \mathfrak{b}$, 
  \begin{align*}
      l_1(f)=~& P[\Omega,f] = [\nu_{_{\rightharpoonup,\leftharpoonup}},f], \\
  l_2(f,g)=~&P[[\Omega,f],g]=[[\mu_{_{\vartriangleright,\vartriangleleft}},f],~g],  \\
 l_3(f,g,h)=~& P[[[\Omega,f],g],h]= [[[\widetilde{\theta},f],g],h].
      \end{align*}
      Moreover, for any $f, g, h \in \mathfrak{b}$, one can show that $[[[\Omega,f],g],h] \in \mathfrak{b}$. Since $\mathfrak{b}$ is an abelian Lie subalgebra, it follows that
\begin{align*}
    l_k (f_1, \ldots, f_k) =P [ \cdots [[ \Omega, f_1], f_2], \ldots, f_k] = 0,
\end{align*}
for $k \geq 4$ and $f_1, \ldots, f_k \in \mathfrak{b}$.
This proves the first part.

Finally, for any linear map $r:B\rightarrow A$ viewed as an element $r\in \mathfrak{b}_0 = \mathrm{Hom}(B,A)$, we have 
\begin{align*}
    &\big(\displaystyle\sum_{k=1}^{\infty} \dfrac{1}{k!}~l_k(r,\ldots,r)\big) (x,y)\\
    &= l_1(r)(x,y) + \frac{1}{2} l_2(r,r)(x,y)+ \frac{1}{6} l_3(r,r,r) (x,y)\\
    &= [\nu_{_{\rightharpoonup,\leftharpoonup}},r](x,y) + \frac{1}{2} [[\mu_{_{\vartriangleright,\vartriangleleft}},r],r](x,y) + \frac{1}{6} [[[\widetilde{\theta},r],r],r](x,y)\\
    &= -r(x\cdot_B y)+ r(x)\leftharpoonup y + x\rightharpoonup r(y) + r(x)\cdot_A r(y)- r\big( r(x)\vartriangleright y + x\vartriangleleft r(y)\big) - r\big(\theta\big(r(x),r(y)\big)\big).
\end{align*}
This shows that $r$ is a Maurer-Cartan element of the $L_{\infty}$-algebra $(\mathfrak{b},\{l_k\}_{k=1}^{\infty})$ if and only if $r$ is a weak deformation map.
\end{proof}

The $L_\infty$-algebra $(\mathfrak{b}, \{ l_k \}_{k=1}^\infty)$ is called the {\bf controlling algebra} for weak deformation maps in the given quasi-twilled associative algebra. This generalizes the already existing controlling algebras for Rota-Baxter operators of weight $0$ \cite{das-rota}, Rota-Baxter operators of weight $1$ \cite{das-weighted}, twisted Rota-Baxter operators \cite{das-ns}, Reynolds operators \cite{das-ns}, left averaging operators \cite{wang-zhou} and right averaging operators \cite{wang-zhou}. As a new application, if we consider the quasi-twilled associative algebra given in Example \ref{matched-quasi}, we get the controlling algebra for deformation maps in a matched pair of associative algebras.

\begin{prop}
    Let $(A_\mu, B_\nu, \vartriangleright, \vartriangleleft, \rightharpoonup, \leftharpoonup)$ be a matched pair of associative algebras. Then the pair $( \oplus_{n=0}^\infty \mathrm{Hom} (B^{\otimes n+1}, A), \{ l_k \}_{k=1}^\infty) $ is an $L_\infty$-algebra, where $l_k = 0$ for $k \neq 1, 2$ and
    \begin{align*}
        l_1 (f) (x_1, \ldots, x_{m+1}) =~&(-1)^{m+1} ~ \! x_1 \rightharpoonup f(x_2, \ldots, x_{m+1}) + f(x_1, \ldots, x_m) \leftharpoonup x_{m+1} \\
        ~&+ \sum_{i=1}^m (-1)^{i+m+1} ~ \! f (x_1, \ldots, x_{i-1}, x_i \cdot x_{i+1}, \ldots, x_{m+1}),
    \end{align*}
    \begin{align*}
        l_2 (f, g) &(x_1, \ldots, x_{m+n}) =   (-1)^m \sum_{i=1}^m (-1)^{(i-1)n} ~ \! f \big(   x_1, \ldots, x_{i-1}, g(x_i, \ldots, x_{i+n-1}) \vartriangleright x_{i+n}, \ldots, x_{m+n}   \big) \\
        &- (-1)^m \sum_{i=1}^m (-1)^{in} f \big(  x_1, \ldots, x_{i-1}, x_i \vartriangleleft g (x_{i+1}, \ldots, x_{i+n}), x_{i+n+1}, \ldots, x_{m+n}   \big) \\
        &- (-1)^{ m (n+1)} \bigg\{ \sum_{i=1}^n (-1)^{(i-1)m} ~ \! g \big(   x_1, \ldots, x_{i-1}, f(x_i, \ldots, x_{i+m-1}) \vartriangleright x_{i+m}, \ldots, x_{m+n}   \big) \\
        & \qquad \qquad \qquad \quad  - \sum_{i=1}^n (-1)^{im} g \big(  x_1, \ldots, x_{i-1}, x_i \vartriangleleft f (x_{i+1}, \ldots, x_{i+m}), x_{i+m+1}, \ldots, x_{m+n}   \big) \bigg\} \\
        &+ (-1)^{m (n+1)} f(x_1, \ldots, x_m ) \cdot_A g (x_{m+1}, \ldots, x_{m+n}) - (-1)^m ~ \! g(x_1, \ldots, x_n) \cdot_A f (x_{n+1}, \ldots, x_{m+n}),
    \end{align*}
    for $f \in \mathrm{Hom} (B^{\otimes m} , A)$ and $g \in \mathrm{Hom} (B^{\otimes n} , A)$. Moreover, a linear map $r: B \rightarrow A$ is a deformation map in $(A_\mu, B_\nu, \vartriangleright, \vartriangleleft, \rightharpoonup, \leftharpoonup)$ in the sense of (\ref{agore-defor}) if and only if $r$ is a Maurer-Cartan element of this $L_\infty$-algebra.
\end{prop}

In Theorem \ref{Tmcla}, we have seen that a Maurer-Cartan element of a (curved) $L_{\infty}$-algebra gives rise to a new $L_{\infty}$-algebra structure by twisting. By applying this result to the present context, we get the following.
\begin{thm}\label{Tca}
    Let $(A_\mu, B_\nu, \vartriangleright, \vartriangleleft, \rightharpoonup, \leftharpoonup,\theta)$ be a quasi-twilled associative algebra and $r:B\rightarrow A$ be a weak deformation map. Then $\big(\mathfrak{b} = \oplus_{n=0}^\infty\mathrm{Hom} (B^{\otimes n+1}, A),\{l_k^r\}_{k=1}^{\infty}\big)$ is an $L_{\infty}$-algebra, where for $f, g, h \in \mathfrak{b}$,
    \begin{align*} 
     l_1^r (f) =~& l_1(f) + l_2(r,f)+ \frac{1}{2} l_3(r,r,f),\\
     l_2^r(f,g) =~& l_2(f,g) + l_3(r,f,g),\\
     l_3^r(f,g,h) =~& l_3(f,g,h),\\
     l_k^r =~& 0,~~\mathrm{for~} k\geq 4,
    \end{align*}
    where $l_1,l_2,l_3$ are given by (\ref{r-l1})-(\ref{r-l3}). Moreover, for any linear map $r':B\rightarrow A$, the sum $r+r':B\rightarrow A$ is also a weak deformation map in the given quasi-twilled associative algebra if and only if $r' \in \mathfrak{b}_0$ is a Maurer-Cartan element of the $L_{\infty}$-algebra $(\mathfrak{b},\{l_k^r\}_{k=1}^{\infty})$.
\end{thm}
\begin{proof}
    The first part follows from Theorem \ref{Tmcla}. For the second part, we observe that \begin{align*}
   & l_1(r+r') + \frac{1}{2!} l_2(r+r',r+r')+ \frac{1}{3!} l_3(r+r',r+r',r+r')\\
    &=\underbrace{l_1(r) + \frac{1}{2!} l_2(r,r)+ \frac{1}{3!} l_3(r,r,r)}_{=0} + \big\{l_1(r') +  l_2(r,r')+ \frac{1}{2} l_3(r,r,r') \\
    & \quad + \frac{1}{2} l_2(r',r')+ \frac{1}{2} l_3(r,r',r') + \frac{1}{6} l_3(r',r',r')\big\}\\
    &= l_1^r(r') + \frac{1}{2!} l_2^r(r',r')+ \frac{1}{3!} l_3^r(r',r',r').
  \end{align*} 
  This shows that $r+r'$ is a weak deformation map if and only if $r'$ is a Maurer-Cartan element of the $L_{\infty}$-algebra $(\mathfrak{b},\{l_k^r\}_{k=1}^{\infty})$.
\end{proof}
The $L_{\infty}$-algebra $\big(\mathfrak{b} = \oplus_{n=0}^\infty\mathrm{Hom} (B^{\otimes n+1}, A),\{l_k^r\}_{k=1}^{\infty}\big)$ constructed in the above theorem is called the \textbf{governing algebra} of the deformation map $r$. This name is justified by the last part of Theorem \ref{Tca}.

    Let $r:B\rightarrow A$ be a weak deformation map in the given quasi-twilled associative algebra.
    %$(A_\mu, B_\nu, \vartriangleright, \vartriangleleft, \rightharpoonup, \leftharpoonup,\theta)$. 
    Since $(\mathfrak{b},\{l_k^r\}_{k=1}^{\infty})$ is an $L_{\infty}$-algebra, it follows that the degree $1$ map $l^r_1:\mathfrak{b}\rightarrow \mathfrak{b}$ is a differential. In the following result, we show that the map $l^r_1$ coincides with the map $\delta^r$ given in (\ref{hoch-r-weak}) up to some sign.
    
    \begin{prop}
     %Let $r:B\rightarrow A$ be a weak deformation map in the quasi-twilled associative algebra $(A_\mu, B_\nu, \vartriangleright, \vartriangleleft, \rightharpoonup, \leftharpoonup,\theta)$. 
     For any $f\in \mathfrak{b}_{n-1} = \mathrm{Hom}(B^{\otimes n},A)$, we have \begin{align*}
         l_1^r(f) = (-1)^{n-1} ~ \! \delta^r(f).
     \end{align*}   
    \end{prop}
    \begin{proof}
        Let $(a,x),(b,y)\in \mathbb{A}$ be arbitrary. Then by straightforward calculations, we see that 
        \begin{align*}
        [\mu_{_{\vartriangleright,\vartriangleleft}},r]((a,x),(b,y))=\big(r(x)\cdot_A b+ a\cdot_A r(y) -r(x\vartriangleleft b)-r(a\vartriangleright y), r(x)\vartriangleright y+ x\vartriangleleft r(y)\big)
        \end{align*}
        and \begin{align*}
            \frac{1}{2}[[\widetilde{\theta},r],r] ((a,x),(b,y))= \big(-r\big(\theta(r(x),b)\big)-r\big(\theta(a,r(y))\big), ~ \! \theta(r(x),r(y)) \big).
        \end{align*}
        This implies that $\nu_{_{\rightharpoonup,\leftharpoonup}}+[\mu_{_{\vartriangleright,\vartriangleleft}},r]+\frac{1}{2}[[\widetilde{\theta},r],r] = \nu_r+\rightharpoonup_r+\leftharpoonup_r$ (using the notations from (\ref{a6}), (\ref{a9}), (\ref{a10})).
        Hence we have \begin{align*}
         l_1^r (f) =~& l_1(f) + l_2(r,f)+ \frac{1}{2} l_3(r,r,f),\\ =~&
         [\nu_{_{\rightharpoonup,\leftharpoonup}},f]+[[\mu_{_{\vartriangleright,\vartriangleleft}},r],f]+ \frac{1}{2}[[[\widetilde{\theta},r],r],f]\\
         =~& [\nu_{_{\rightharpoonup,\leftharpoonup}}+[\mu_{_{\vartriangleright,\vartriangleleft}},r]+\frac{1}{2}[[\widetilde{\theta},r],r],f]\\
         =~& [\nu_r+\rightharpoonup_r+\leftharpoonup_r,f] =(-1)^{n-1} ~ \! \delta^r(f).
        \end{align*}
        This proves the result.
    \end{proof}
\subsection{Simultaneous deformations of a quasi-twilled associative algebra and a weak deformation map} In this subsection, we study the governing algebra and deformations of the pair $(\mathbb{A}, r)$ consisting of a quasi-twilled associative algebra $\mathbb{A}$ and a fixed weak deformation map $r$.

Let $A$ and $B$ be two vector spaces and
   $ \mathfrak{g}= \big(\oplus_{n=0}^\infty \mathrm{Hom} \big((A\oplus B)^{\otimes n+1}, A\oplus B \big),[~,~]\big)$ be the graded Lie algebra on the space of all multilinear maps on $A \oplus B$. We have already seen that $\mathfrak{b}= \oplus_{n=0}^\infty \mathrm{Hom} (B^{\otimes n+1}, A)$ is an abelian Lie subalgebra of $\mathfrak{g}$ and if $P:\mathfrak{g}\rightarrow \mathfrak{b}$ is the projection map then $\mathrm{ker}(P)\subset \mathfrak{g}$ is a graded Lie subalgebra. Hence the quadruple $(\mathfrak{g},\mathfrak{b}, P,\Delta=0)$ is a $V$-data.
 % Explicitly, \begin{center}
 %     $ \mathfrak{g}^n:= C^{n+1\lvert -1} \oplus C^{n\lvert 0} \oplus \cdots \oplus C^{0\lvert n}\oplus C^{-1\lvert n+1}$
 %     \end{center}
  %    and let \begin{center}
   %        $\mathscr{Q}^n:= C^{n+1\lvert -1} \oplus C^{n\lvert 0}\oplus C^{n-1\lvert 1} \oplus \cdots \oplus C^{1\lvert n-1} \oplus C^{0\lvert n}$
    %       \end{center}
        %   Then we have seen in Proposition \ref{Pmq} that $\mathscr{Q}=~\oplus_{n=0}^\infty\mathscr{Q}^n$ is a graded Lie subalgebras of $\mathfrak{g}$.
           Therefore, by applying Theorem \ref{Tcvd} (ii) to the graded Lie subalgebra $\mathscr{Q} \subset \mathfrak{g}$, we obtain the following.
           %an $L_{\infty}$-algebra structure on the graded vector space $\mathscr{Q}[1]\oplus \mathfrak{b}$. The structure maps of this $L_{\infty}$-algebra are given in the following result.
           \begin{thm}
                Let $A$ and $B$ be two vector spaces. Then $(\mathscr{Q}[1]\oplus \mathfrak{b},\{\widetilde{l}_k\}_{k=1}^{\infty})$ is an $L_{\infty}$-algebra, where 
                \begin{align*}
                    \widetilde{l}_2(F[1],G[1]) =~& (-1)^{\lvert F \rvert} [F,G][1] ,\\
\widetilde{l}_k\big(F[1],f_1,\ldots,f_{k-1}\big) =~& P[\cdots [[F,f_1],f_2],\ldots,f_{k-1} ],~~k\geq 2,
                \end{align*}
                for all homogeneous elements $F,G\in \mathscr{Q}$ and $f_1,\ldots, f_{k-1}\in \mathfrak{b}$.
           \end{thm}
 Next, suppose there are bilinear maps  
  \begin{align*} 
    &\mu: A \times A\rightarrow A, \qquad  \nu: B \times B\rightarrow B,\\
    \vartriangleright ~: A \times B\rightarrow B, \quad \vartriangleleft ~ : B \times A & \rightarrow B, \quad  \rightharpoonup ~: B \times A\rightarrow A, \quad \leftharpoonup ~ : A \times B\rightarrow A, \quad \theta: A \times A\rightarrow B
    \end{align*}
    and a linear map $r:B \rightarrow A$. We consider the element 
      $\Omega =\widetilde{\theta}+\mu_{_{\vartriangleright,\vartriangleleft}}+\nu_{_{\rightharpoonup,\leftharpoonup}}\in C^{2\lvert -1} \oplus C^{1\lvert 0} \oplus C^{0\lvert 1} = \mathscr{Q}_1.$
 %   Hence \begin{align*}
  %      \alpha =~(\Omega[1],r)\in (\mathscr{Q}[1]\oplus \mathfrak{b})_0
 %  \end{align*}
%is a degree 0 element of the graded vector space $\mathscr{Q}[1]\oplus \mathfrak{b}$. With the above notations, we have the following result.

\begin{thm}
 With the above notations, the tuple $(A_\mu, B_\nu, \vartriangleright, \vartriangleleft, \rightharpoonup, \leftharpoonup,\theta)$ is a quasi-twilled associative algebra and $r: B\rightarrow A$ is a weak deformation map in it if and only if the element $\alpha =(\Omega[1],r)\in (\mathscr{Q}[1]\oplus \mathfrak{b})_0$ is a Maurer-Cartan element of the $L_{\infty}$-algebra $(\mathscr{Q}[1]\oplus \mathfrak{b},\{ \widetilde{l}_k\}_{k=1}^{\infty})$.   
\end{thm}
\begin{proof}
It follows from the definition of the structure maps that $\widetilde{l}_1\big((\Omega[1],r)\big) = 0$. On the other hand, using the bidegree reason, one can show that $[[[\Omega,r],r],r] \in \mathfrak{b}$ and therefore,
\begin{align*}
    [[[[\Omega,r],r],r],r] =~0
\end{align*} 
as $\mathfrak{b}$ is an abelian Lie subalgebra.
As a result, we get that $\widetilde{l}_k \big((\Omega[1],r),\ldots,(\Omega[1],r) \big)=~0,$ for all $k\geq 5$. Hence 
\begin{align*}
    &\displaystyle\sum_{k=1}^{\infty} \dfrac{1}{k!}~\widetilde{l}_k\big((\Omega[1],r),\ldots,(\Omega[1],r)\big)\\
 &=  \frac{1}{2!} ~ \! \widetilde{l}_2\big((\Omega[1],r),(\Omega[1],r)\big)+ \frac{1}{3!} ~ \! \widetilde{l}_3\big((\Omega[1],r),(\Omega[1],r),(\Omega[1],r)\big) + \frac{1}{4!} ~ \! \widetilde{l}_4\big((\Omega[1],r),\ldots,(\Omega[1],r)\big) \\  
&= \frac{1}{2!} \big\{ \widetilde{l}_2(\Omega[1],\Omega[1])+ 2 ~ \! 
 \widetilde{l}_2(\Omega[1],r)\big\} + \frac{1}{3!} \big\{3 ~ \! \widetilde{l}_3(\Omega[1],r,r)\big\} + \frac{1}{4!}  \big\{4 ~ \! \widetilde{l}_4(\Omega[1],r,r,r)\big\}\\
&= \big(-\frac{1}{2} [\Omega,\Omega][1], ~ \! P[\Omega,r]+ \frac{1}{2!} P[[\Omega,r],r] + \frac{1}{3!} P[[[\Omega,r],r],r] \big)\\
&= \big(-\frac{1}{2} [\Omega,\Omega][1], ~ \! l_1 (r) + \frac{1}{2} l_2 (r,r) + \frac{1}{6} l_3 (r,r,r) \big),
\end{align*}
where the maps $l_1, l_2, l_3$ are respectively given in (\ref{r-l1})-(\ref{r-l3}).
We know that $[\Omega,\Omega]=0$ if and only if the tuple $(A_\mu, B_\nu, \vartriangleright, \vartriangleleft, \rightharpoonup, \leftharpoonup, \theta )$ is a quasi-twilled associative algebra. On the other hand, we have seen in Theorem \ref{Tqtallar} that $l_1 (r) + \frac{1}{2} l_2 (r,r) + \frac{1}{6} l_3 (r,r,r) = 0$ if and only if 
 $r: B \rightarrow A$ is a weak deformation map in the above quasi-twilled associative algebra. Hence the result follows.
\end{proof}

Next, let $(\mathbb{A}, r)$ be a pair of a quasi-twilled associative algebra $\mathbb{A}$ with a fixed decomposition $\mathbb{A} = A \oplus B$ and a weak deformation map $r$ in it. Then by applying Theorem \ref{Tmcla}, we get the following result.

\begin{thm}
   Let $(\mathbb{A}, r)$ be a pair of a quasi-twilled associative algebra $\mathbb{A}$ with a fixed decomposition $\mathbb{A} = A \oplus B$ and a weak deformation map $r$ in it. Then the pair $\big(  \mathscr{Q}[1] \oplus \mathfrak{b} , \{  \widetilde{l}_k^{ (\Omega [1], r)}  \}_{k=1}^\infty \big)$ is an $L_\infty$-algebra, where 
   \begin{align*}
        \widetilde{l}_k^{ (\Omega [1], r)} \big(  (F_1 [1], f_1), \ldots, (F_k [1], f_k) \big) := \sum_{n=0}^\infty \frac{1}{n!} ~ \! \widetilde{l}_{n+k} \big( (\Omega [1], r), \ldots,  (\Omega [1], r) , (F_1 [1], f_1), \ldots, (F_k [1], f_k)     \big),
   \end{align*}
   for $k \geq 1$ and homogeneous elements $(F_1 [1], f_1), \ldots, (F_k [1], f_k) \in \mathscr{Q}[1] \oplus \mathfrak{b}$. Moreover, for any other element $\Omega' \in \mathscr{Q}_1$ and a linear map $r': B \rightarrow A$, the sum $\Omega + \Omega' \in \mathscr{Q}_1$ represents a new quasi-twilled associative algebra structure on $\mathbb{A}$ and $r + r': B \rightarrow A$ is a weak deformation map in this new structure if and only if $(\Omega' [1], D') \in (\mathscr{Q}[1]\oplus \mathfrak{b})_0$ is a Maurer-Cartan element of  $  \big(  \mathscr{Q}[1]\oplus \mathfrak{b},\{ \widetilde{l}_k^{   ( \Omega [1], r)  } \}_{k=1}^{\infty}  \big)$.
\end{thm}

The last part of the above theorem says that the $L_\infty$-algebra   $  \big(  \mathscr{Q}[1]\oplus \mathfrak{b},\{ \widetilde{l}_k^{   ( \Omega [1], r)  } \}_{k=1}^{\infty}  \big)$ governs the simultaneous deformations of the given quasi-twilled associative structure on $\mathbb{A}$ and the weak deformation map $r$. For this reason, we call this $L_\infty$-algebra as the {\bf governing algebra} for the pair $(\mathbb{A}, r)$. This generalizes the $L_\infty$-algebras that govern deformations of Rota-Baxter algebras \cite{das-mishra,wang-zhou2} and left (resp. right) averaging algebras \cite{wang-zhou}.

%\textcolor{red}{ we need to write the cohomological relation of $r$, $(A, B, \vartriangleright, \vartriangleleft, \rightharpoonup, \leftharpoonup, \theta )$  and $((A, B, \vartriangleright, \vartriangleleft, \rightharpoonup, \leftharpoonup, \theta ),r)$}

\section{Rota-Baxter operators twisted by non-abelian 2-cocycles and twisted tridendriform algebras}\label{sec8}

In this section, we first consider Rota-Baxter operators twisted by non-abelian $2$-cocycles. Then we introduce a new algebraic structure, called ``twisted tridendriform algebra''. We show that a Rota-Baxter operator twisted by some non-abelian $2$-cocycle induces a twisted tridendriform algebra structure and conversely, any twisted tridendriform algebra arises in this way.

 Let $A$ and $B$ be two associative algebras. Recall that \cite{agore4} a {\bf non-abelian $2$-cocycle} on $A$ with values in $B$ is a triple $(\vartriangleright, \vartriangleleft,\theta)$ consisting of bilinear maps $\vartriangleright ~: A \times B\rightarrow B,~~ \vartriangleleft ~: B \times A\rightarrow B$ and $\theta: A \times A\rightarrow B$ such that the following conditions are hold:
 \begin{align}
          (a\cdot_A b) \vartriangleright x + \theta(a,b) \cdot_B x =~& a \vartriangleright (b\vartriangleright x), \label{nab1}\\
         (a\vartriangleright x) \vartriangleleft b =~& a \vartriangleright (x\vartriangleleft b), \label{nab2}\\
         (x\vartriangleleft a)\vartriangleleft b =~& x\vartriangleleft (a\cdot_A b) + x \cdot_B \theta(a,b), \label{nab3}\\
         (a \vartriangleright x ) \cdot_B y =~& a \vartriangleright (x \cdot_B y), \label{nab4}\\
         (x \vartriangleleft a) \cdot_B y =~& x \cdot_B (a \vartriangleright y), \label{nab5}\\
         (x \cdot_B y) \vartriangleleft a =~& x \cdot_B (y \vartriangleleft a), \label{nab6}\\
        \theta ( a, b) \vartriangleleft c + \theta (a \cdot_A b, c) =~& a \vartriangleright \theta (b, c) + \theta (a, b \cdot_A c), \label{nab7}
          \end{align}
 for any $a,b,c\in A$ and $x , y \in B$.
 Non-abelian 2-cocycles defined above are closely related to non-abelian extensions of the associative algebra $A$ by the algebra $B$.
 
 \begin{remark} \cite{agore4} \label{ex9}
  Let $A$ and $B$ be two associative algebras and $(\vartriangleright, \vartriangleleft,\theta)$ be a non-abelian 2-cocycle on $A$ with values in $B$. Then the direct sum $A\oplus B$ can be given an associative algebra structure with the multiplication  \begin{align*}
      (a,x)\boxtimes (b,y) = (a \cdot_A b,~ x \cdot_B y +a \vartriangleright y + x\vartriangleleft b + \theta(a,b) ),
  \end{align*}
  for $(a,x),(b,y)\in A\oplus B$.
  Then $(A\oplus B,\boxtimes)$ is a quasi-twilled associative algebra.    
 \end{remark}

\begin{defn}
    Let $A, B$ be two associative algebras and $(\vartriangleright, \vartriangleleft, \theta)$ be a non-abelian $2$-cocycle on $A$ with values in $B$. A linear map $r: B \rightarrow A$ is said to be a {\bf $(\vartriangleright, \vartriangleleft, \theta)$-twisted Rota-Baxter operator} if it satisfies
    \begin{align}\label{rbt-nab}
        r(x) \cdot_A r(y) = r \big(  r(x) \vartriangleright y + x \vartriangleleft r (y) + x \cdot_B y + \theta ( r(x), r (y)) \big), \text{ for } x, y \in B.
    \end{align}
\end{defn}

It follows that a $(\vartriangleright, \vartriangleleft, \theta)$-twisted Rota-Baxter operator is nothing but a weak deformation map in the quasi-twilled associative algebra $(A\oplus B,\boxtimes)$ given above.

\begin{defn}\label{defn-tta}
A {\bf twisted tridendriform algebra} is a tuple $(\mathcal{A},\prec,\succ,\curlyvee,\cdot)$ consisting of a vector space $\mathcal{A}$ endowed with four bilinear maps $\prec, \succ, \curlyvee, \cdot:\mathcal{A}\times\mathcal{A} \rightarrow \mathcal{A}$ such that for any $x, y, z \in \mathcal{A}$, the following set of identities are hold:
\begin{align}
    (x \prec y) \prec z =~& x \prec (y \ast z) + x \cdot (y \curlyvee z), \tag{TT1} \label{tt1} \\
    (x \succ y) \prec z =~& x \succ (y \prec z), \tag{TT2} \label{tt2}\\
    (x \star y) \succ z + (x \curlyvee y) \cdot z =~& x \succ (y \succ z), \tag{TT3} \label{tt3}\\
    (x \succ y) \cdot z =~& x \succ (y \cdot z), \tag{TT4} \label{tt4}\\
    (x \prec y) \cdot z =~& x \cdot (y \succ z), \tag{TT5} \label{tt5}\\
    (x \cdot y) \prec z =~& x \cdot (y \prec z), \tag{TT6} \label{tt6}\\
    (x \cdot y)\cdot z =~& x \cdot (y \cdot z), \tag{TT7} \label{tt7}\\
    (x \curlyvee y) \prec z + (x \star y) \curlyvee z =~& x \succ (y \curlyvee z) + x \curlyvee (y \star z), \tag{TT8} \label{tt8}
\end{align}
where $x \star y := x \prec y + x \succ y + x \curlyvee y + x \cdot y$, for any $x, y \in \mathcal{A}.$
\end{defn}

\begin{remark}
\begin{itemize}
    \item[(i)] Let $(\mathcal{A}, \prec, \succ, \curlyvee, \cdot)$ be a twisted tridendriform algebra. Then it follows from (\ref{tt7}) that $(\mathcal{A}, \cdot)$ is an associative algebra. 
    \item[(ii)] In a twisted tridendriform algebra $(\mathcal{A}, \prec, \succ, \curlyvee, \cdot)$, if $\curlyvee$ is trivial then $(\mathcal{A},\prec,\succ,\cdot)$ is a tridendriform algebra considered by Loday and Ronco \cite{loday-ronco}. If $\cdot~ \!$ is trivial then $(\mathcal{A},\prec,\succ,\curlyvee)$ is simply an NS-algebra considered by Uchino \cite{uchino}. Finally, if $\curlyvee$ and $\cdot$ both are trivial then $(\mathcal{A},\prec,\succ)$ becomes a dendriform algebra \cite{loday-dialgebra}.
\end{itemize}    
\end{remark}

%\textcolor{red}{Example}

In the following, we show that Rota-Baxter operators twisted by non-abelian $2$-cocycles induce twisted tridendriform structures.
\begin{prop}\label{Ptta}
 Let $A$ and $B$ be two associative algebras and $(\vartriangleright, \vartriangleleft,\theta)$ be a non-abelian $2$-cocycle on $A$ with values in $B$. Let $r:B\rightarrow A$ be a $(\vartriangleright, \vartriangleleft,\theta)$-twisted Rota-Baxter operator. Then $(B,\prec_r,\succ_r,\curlyvee_r,\cdot_B)$ is a twisted tridendriform algebra, where 
 \begin{align*}
     x \prec_r y:=~   x\vartriangleleft r(y),\quad x\succ_r y:=~ r(x)\vartriangleright y \quad \mathrm{and} \quad x\curlyvee_ry:=~\theta (r(x),r(y)),~~\mathrm{for~}x,y\in B.
 \end{align*}  
\end{prop}

\begin{proof}
Let $x, y, z \in B$. Then we have
\begin{align*}
    (x \prec_r y ) \prec_r z = (x \vartriangleleft r (y)) \vartriangleleft r (z) \stackrel{(\ref{nab3})}{=}~& x \vartriangleleft \big(  r(y) \cdot_A r (z) \big) + x \cdot_B \theta ( r(y), r(z)) \\
    =~& x \prec_r (y \star z) + x \cdot_B (y \curlyvee_r z) \quad (\text{by} ~(\ref{rbt-nab})),
\end{align*}
where $x \star y = x \prec_r y + x \succ_r y + x \curlyvee_r y + x \cdot_B y$, for any $x , y \in B.$ 
This verifies the identity (\ref{tt1}). Similarly, we have
\begin{align*}
    (x \succ_r y) \prec_r z = ( r (x) \vartriangleright y) \vartriangleleft r (z) \stackrel{(\ref{nab2})}{=}~& r(x) \vartriangleright (y \vartriangleleft r(z)) = x \succ_r (y \prec_r z), \\ 
    (x \star y) \succ_r z + (x \curlyvee_r y) \cdot_B z =~& r (x \star y) \vartriangleright z + \theta ( r(x), r (y)) \cdot_B z \\
    =~& ( r(x) \cdot_A r(y)) \vartriangleright z + \theta ( r(x) , r(y)) \cdot_B z \quad (\text{by} ~(\ref{rbt-nab})) \\
    \stackrel{(\ref{nab1})}{=}~& r(x) \vartriangleright ( r(y) \vartriangleright z) = x \succ_r (y \succ_r z),\\
    (x \succ_r y) \cdot_B z = (r(x) \vartriangleright y) \cdot_B z \stackrel{(\ref{nab4})}{=}~& r(x) \vartriangleright (y \cdot_B z) = x \succ_r (y \cdot_B z), \\
    (x \prec_r y) \cdot_B z = (x \vartriangleleft r(y)) \cdot_B z \stackrel{(\ref{nab5})}{=}~& x \cdot_B ( r(y) \vartriangleright z) = x \cdot_B (y \succ_r z),\\
    (x \cdot_B y) \prec_r z = (x \cdot_B y) \vartriangleleft r(z) \stackrel{(\ref{nab6})}{=}~& x \cdot_B (y \vartriangleleft r(z)) = x \cdot_B (y \prec_r z)
\end{align*}
which proves (\ref{tt2})-(\ref{tt6}). The identity (\ref{tt7}) also follows as $(B, \cdot_B)$ is an associative algebra. Finally,
\begin{align*}
    (x \curlyvee_r y) \prec_r z + (x \star y) \curlyvee_r z =~& \theta ( r(x), r(y)) \vartriangleleft r (z) + \theta ( r(x) \cdot_A r(y), r(z) ) \\
    \stackrel{(\ref{nab7})}{=}~& r (x) \vartriangleright \theta (r(y), r(z)) + \theta ( r(x), r(y) \cdot_A r(z)) \\
    =~& x \succ_r (y \curlyvee_r z) + x \curlyvee_r (y \star z)
\end{align*}
which verifies (\ref{tt8}). Hence $(B, \prec_r, \succ_r, \curlyvee_r, \cdot_B)$ is a twisted tridendriform algebra.
\end{proof}

In \cite{ebrahimi} K. Ebrahimi-Fard showed that a relative Rota-Baxter operator of weight $1$ induces a tridendriform algebra structure. On the other hand, Uchino \cite{uchino} showed that a twisted Rota-Baxter operator induces an NS-algebra structure. It is important to remark that Proposition \ref{Ptta} unifies both of these results.

In the following, we prove the converse of Proposition \ref{Ptta}. More precisely, we show that any twisted tridendriform algebra is always induced by a Rota-Baxter operator twisted by some non-abelian 2-cocycle. 
\begin{thm}
 Let $(\mathcal{A},\prec,\succ,\curlyvee,\cdot)$ be a twisted tridendriform algebra. \begin{itemize}
     \item [(i)] Then $(\mathcal{A},\star)$ is an associative algebra, where \begin{align*}
         x\star y:=~ x\prec y+x \succ y+x\curlyvee y+ x\cdot y,~~\mathrm{for~} x, y \in \mathcal{A}.
         \end{align*}
      \item [(ii)] We define bilinear maps $\vartriangleright,\vartriangleleft,\theta: \mathcal{A} \times \mathcal{A}\rightarrow \mathcal{A}$ by 
      \begin{align*}
          x \vartriangleright y := x \succ y, \quad x \vartriangleleft y := x \prec y ~~~~ \text{ and } ~~~~ \theta (x, y) := x \curlyvee y, \text{ for } x, y \in \mathcal{A}.
      \end{align*}
      Then $(\vartriangleright, \vartriangleleft,\theta)$ is a non-abelian $2$-cocycle on the associative algebra $(\mathcal{A}, {\star})$ with values in the algebra $(\mathcal{A},{\cdot})$.
     \item [(iii)] The identity map $\mathrm{Id}: (\mathcal{A}, \cdot) \rightarrow (\mathcal{A}, \star)$ is a $(\vartriangleright, \vartriangleleft,\theta)$-twisted Rota-Baxter operator. Moreover, the induced twisted
     tridendriform algebra structure on $(\mathcal{A}, {\cdot})$ coincides with the given one. 
 \end{itemize}   
\end{thm}

\begin{proof}
 (i) Since $(\mathcal{A}, \prec, \succ, \curlyvee, \cdot)$ is a twisted tridendriform algebra, the identities (\ref{tt1})-(\ref{tt8}) hold. Adding the left-hand sides of these identities, one obtains $(x \star y ) \star z$. On the other hand, by adding the right-hand sides of the same identities, one gets $x \ast ( y \ast z)$. Hence the result follows.    

 (ii) It is straightforward to see that the seven identities (listed in (\ref{nab1})-(\ref{nab7}))  of a non-abelian $2$-cocycle are respectively equivalent to the identities (\ref{tt3}),  (\ref{tt2}), (\ref{tt1}), (\ref{tt4}), (\ref{tt5}), (\ref{tt6}) and (\ref{tt8}).

 (iii) For any $x, y \in \mathcal{A}$, we have
 \begin{align*}
     \mathrm{Id} (x) \star  \mathrm{Id} (y) =~& x \star y = x \prec y + x \succ y + x \curlyvee y +  x \cdot y \\
     =~&  \mathrm{Id} \big(  x \vartriangleleft  \mathrm{Id} (y) +  \mathrm{Id} (x) \vartriangleright y + \theta (  \mathrm{Id} (x),  \mathrm{Id} (y)) + x \cdot y  \big).
 \end{align*}
 This shows that the identity map $\mathrm{Id} : (\mathcal{A}, \cdot) \rightarrow (\mathcal{A}, \star)$ is a $(\vartriangleright, \vartriangleleft, \theta)$-twisted Rota-Baxter operator.

 If $(\mathcal{A}, \prec_\mathrm{Id}, \succ_\mathrm{Id}, \curlyvee_\mathrm{Id}, \cdot)$ is the induced twisted tridendriform structure then
 \begin{align*}
     x \prec_\mathrm{Id} y = x \vartriangleleft \mathrm{Id}(y) = x \prec y, \quad x \succ_\mathrm{Id} y = \mathrm{Id} (x) \vartriangleright y = x \succ y \\
     \text{ and } ~ x \curlyvee_\mathrm{Id} y = \theta ( \mathrm{Id} (x) , \mathrm{Id} (y)) = x \curlyvee y, \text{ for all } x, y \in \mathcal{A}. \quad
 \end{align*}
 Hence $(\mathcal{A}, \prec_\mathrm{Id}, \succ_\mathrm{Id}, \curlyvee_\mathrm{Id}, \cdot) = (\mathcal{A}, \prec, \succ, \curlyvee, \cdot)$ which proves the final result.
\end{proof}

The above result unifies that a tridendriform algebra is always induced by a relative Rota-Baxter operator of weight 1 and an NS-algebra is always induced by a twisted Rota-Baxter operator.

\begin{remark}
    Let $(\mathcal{A}, \prec, \succ, \curlyvee, \cdot)$ be a twisted tridendriform algebra. Then the direct sum $\mathcal{A} \oplus \mathcal{A}$ carries an associative algebra structure with the multiplication
    \begin{align*}
        (x, x') \ast (y, y') := \big(  x \star y, ~ \! x' \cdot y' + x \succ y' + x' \prec y + x \curlyvee y \big),
    \end{align*}
    for all $(x, x') , (y, y') \in \mathcal{A} \oplus \mathcal{A}$. Then $(\mathcal{A} \oplus \mathcal{A}, \ast)$ is a quasi-twilled associative algebra.
\end{remark}

\medskip

\noindent {\bf Future works.}  (i) It follows from Definition \ref{defn-tta} that the operad $TwTridend$ that encodes twisted tridendriform algebras is a nonsymmetric operad. In \cite{loday-ronco} and \cite{das-ns} the authors have respectively defined the cohomology of tridendriform algebras and NS-algebras. In a subsequent paper, we aim to generalize their constructions and define the cohomology of twisted tridendriform algebras. To define the explicit cochain groups and the coboundary map, one must first understand the nonsymmetric operad $TwTridend$ and its dual operad $(TwTridend)^!$. Since
    \begin{align*}
        \mathrm{dim }~\!(TwTridend)^!(2) =~& \text{ the number of binary operations in a twisted tridendriform algebra} = 4, \\
        \mathrm{dim }~\!(TwTridend)^!(3) =~& \text{ the number of identities in a twisted tridendriform algebra} = 8,
    \end{align*}
    we expect that $\mathrm{dim }~\!(TwTridend)^!(n) = 2^n$, for any $n \geq 2$. We will also explore this in our upcoming paper.

    \medskip

    \medskip

    (ii) We have seen that a twisted tridendriform algebra can be realized as the underlying structure of a Rota-Baxter operator twisted by a non-abelian $2$-cocycle. In a work in progress, we introduce the notion of a ``twisted post-Lie algebra'' as the Lie analogue of a twisted tridendriform algebra. We observe that twisted post-Lie algebras arise from Rota-Baxter operators twisted by non-abelian Lie algebra $2$-cocycles. We further introduce twisted Rota-Baxter operators on (Lie) groups and find their relations to twisted post-Lie algebras.

%\textcolor{red}{Agore's paper: some motivations and examples from their papers}

%\textcolor{red}{controlling, governing algebras for modified r matrices and deformation maps in matched pair of associative algebras}

\medskip

\medskip

\noindent {\bf Acknowledgements.} Both authors thank the Department of Mathematics, IIT Kharagpur for providing the beautiful academic atmosphere where the research has been carried out.
%\vspace*{1cm}

\medskip

\noindent {\bf Funding.} Ramkrishna Mandal would like to thank the Government of India for supporting his work through the Prime Minister Research Fellowship.

\medskip

\noindent {\bf Data Availability Statement.} Data sharing does not apply to this article as no new data were created or analyzed in this study.

\end{document}